\documentclass[12pt,reqno]{amsart}
\usepackage{amsmath,amssymb,latexsym,esint,cite,mathrsfs}
\usepackage{verbatim,wasysym}
\usepackage[left=2.45cm,right=2.45cm,top=2.65cm,bottom=2.65cm]{geometry}
\usepackage{tikz,enumitem,graphicx, subfig, microtype, color}
\usepackage{epic,eepic}
\usepackage{amsmath}
\allowdisplaybreaks
\usepackage[colorlinks=true,urlcolor=blue, citecolor=red,linkcolor=blue,
linktocpage,pdfpagelabels, bookmarksnumbered,bookmarksopen]{hyperref}
\usepackage[hyperpageref]{backref}
\usepackage[english]{babel}

\makeatletter
\@namedef{subjclassname@2020}{%
	\textup{2020} Mathematics Subject Classification}
\makeatother

\numberwithin{equation}{section}

\newtheorem{thm}{Theorem}[section]
\newtheorem{lem}[thm]{Lemma}
\newtheorem{cor}[thm]{Corollary}
\newtheorem{Prop}[thm]{Proposition}
\newtheorem{Def}[thm]{Definition}

\newcommand{\R}{\mathbb{R}}

\def\al {\alpha}

\def\lam {\lambda}

\begin{document}
	
	\title[Stein-Weiss type inequality]{Stein-Weiss type inequality on the upper \\ half space and its applications}
	
	\author[X.\ Li]{Xiang Li}
	\author[Z.\ Shen]{Zifei Shen}
	\author[M.\ Squassina]{Marco Squassina}
	\author[M.\ Yang]{Minbo Yang}

	\address{Xiang Li, Zifei Shen, Minbo Yang \newline\indent Department of Mathematics, Zhejiang Normal University, \newline\indent
		Jinhua, Zhejiang, 321004, People's Republic of China}
	\email{X. Li: xiangli@zjnu.edu.cn}
		\email{Z. Shen: szf@zjnu.edu.cn}
	\email{M. Yang: mbyang@zjnu.edu.cn}
	
\address{Marco Squassina \newline\indent College of Science
	\newline\indent
	Princess Nourah Bint Abdul Rahman University
	\newline\indent
	Saudi Arabia, Riyadh, PO Box 84428}
\email{marsquassina@pnu.edu.sa}
\address{Dipartimento di Matematica e Fisica
\newline\indent
Universit\`a Cattolica del Sacro Cuore
\newline\indent
Via dei Musei 41, Brescia, Italy}
\email{marco.squassina@unicatt.it}

	\subjclass[2020]{35B40, 45P05, 35B53}
	\keywords{Stein-Weiss type inequality; The method of moving plane; Classification; Extremal functionals}
	
		\thanks{Minbo Yang is partially supported by NSFC (11971436, 12011530199) and ZJNSF (LZ22A010001, LD19A010001).}
	
	\begin{abstract}
In this paper, we establish some Stein-Weiss type inequalities with general kernels on the upper half space and study the existence of extremal functions for this inequality with the optimal constant. Furthermore, we also investigate the regularity, asymptotic estimates, symmetry and non-existence results of the positive solutions of the corresponding Euler-Lagrange integral system.  As an application, we finally derive some Liouville type results for the Hartree type equations in the half space.
\end{abstract}
	
	\maketitle
	
%

	\section{Introduction and main results}
The research of inequality problems have received a great deal of attention in recent years. The inequalities such as Sobolev inequality, Hardy-Littlewood-Sobolev inequality (HLS for short) and Stein-Weiss type inequality play an essential role in the theory of partial differential equations and geometric analysis. For instance, the HLS inequality is widely used to investigate the qualitative properties and classification of solutions. Moreover, the HLS inequality means that sharp Sobolev inequality, as well as Gross's logarithmic Sobolev inequality, which are key ingredients in the study of Yamabe problem and Ricci flow problem \cite{Gross}. Let us first review the classical version of Hardy-Littlewood-Sobolev inequality, which is found by Hardy, Littlewood and Sobolev in \cite{Hardy21,Sobolev21}.

\begin{Prop}
Let $1<p,~q^{\prime}<\infty$, $0<\mu<n$, $f\in L^{p}(\R^n)$, and $g\in  L^{q^{\prime}}(\R^n)$. Then there exists a sharp constant $C\left(p,q^{\prime}, \mu, n\right)$ such that
\begin{equation}\label{HLs}
	\int_{\R ^{n}}\int_{\R^n}\frac{f(y)g(x)}{|x-y|^{\mu}}dxdy\leq C(p,q^{\prime},n,\mu)\|f\|_{p}\|g\|_{q^{\prime}},
\end{equation}
where $\frac{1}{p}+\frac{1}{q^{\prime}}+\frac{\mu}{n}=2$ and $C(p,q^{\prime},n,\mu)$ is independent of $f$ and $g$.
Moreover, the optimal constant satisfy
\begin{equation}\nonumber
C\left(p,q^{\prime}, \mu, n\right)\leq
\frac{n}{(n-\mu)}\left(|\mathbb{S}^{n-1}|/n\right)^{\frac{\mu}{n}}\frac{1}{pq^{\prime}}\left(\left(\frac{\mu/n}{1-1/p}\right)^{\mu/n}+\left(\frac{\mu/n}{1-q^{\prime}}\right)^{\mu/n}\right),
\end{equation}
where we use $q^{\prime}$ to stand for the dual index of $q$.
\end{Prop}
 If one of $p$ or $q^{\prime}$ equals $2$ or $p=q^{\prime}$, the existence of extremals to the HLS inequality with the optimal constant was discussed by Lieb \cite{Lieb1}. Moreover, Lieb considered the value of the sharp constant to inequality \eqref{HLs}. However, for $p\not=q^{\prime}$, either the sharp constant or the extremals is not known. Later, with the aid of the reflection positivity of inversions in spheres, Frank and Lieb \cite{FrankL} explored the best constant and extremals to this inequality in the special case $p=q^{\prime}=\frac{2n}{2n-\mu}$. It was also studied by Carlen and Loss \cite{CarlenLoss} via the competing symmetry argument. For the case $\mu=n-2$, Carlen et al. \cite{Carlen1} simplified the proof and obtained the sharp version of inequality \eqref{HLs} with $n\geq3$. More generally, Frank and Lieb \cite{FrankL2} considered the sharp constant of inequality \eqref{HLs} under the more general parameters by using the rearrangement argument. The HLS inequality on the Heisenberg groups was firstly established in \cite{Folland} by Folland and Stein, and the authors in \cite{FrankL1} studied the sharp version of this inequality in the special case. Dou and Zhu \cite{DouZ} investigated the reversed HLS inequality.  The reversed HLS inequality was proved by Dou, Guo and Zhu \cite{DouGZ} via the subcritical argument. By virtue of the layer cake representation, Ng$\hat{o}$ and Nguyen \cite{Ngo} employed a new method to prove the reversed HLS inequality. 

 Stein and Weiss in \cite{ESGW,ELML} established the Stein-Weiss type inequality,
\begin{equation}\label{whls}
\int_{\R ^{n}}\int_{\R^n}\frac{f(y)g(x)}{|x|^{\al}|x-y|^{\mu}|y|^{\beta}}dxdy\leq C(p,q^{\prime},n,\mu,\al, \beta)\|f\|_{p}\|g\|_{q^{\prime}},
\end{equation}
where $p,q^{\prime}>1$, $0<\mu<n$, $\alpha+\beta\geq0$ and the double weights also satisfies $\frac{1}{p}+\frac{1}{q^{\prime}}+\frac{\al+\beta+\mu}{n}=2$ and $1-\frac{1}{p}-\frac{\mu}{n}<\frac{\al}{n}<1-\frac{1}{p}$. It is well known that, for the case $p<q$, Lieb \cite{Lieb1} proved that the extremals for this inequality can be obtained with the restriction $\alpha,\beta\geq0$, and he also studied the non-existence of extremal functions for inequality \eqref{whls} in the case $p=q$. Moreover, the same result as Lieb's was obtained by Herbst \cite{Herbst} in the case $\mu=n-1,p=q=2,\alpha=0$ and $\beta=1$. While in \cite{Beckner1, Beckner2}, Beckner established a new equivalent formulation to study the best constant of the inequality with $p=q$. For the case $p\neq q$, the author further constructed the estimates of inequality \eqref{whls} with the best constant in \cite{Beckner2}. Later, by applying the concentration compactness principle, Chen, Lu and Tao \cite{CLTh} classified the extremal functions for inequality \eqref{whls} under the hypothesis $p<q$ and $\alpha+\beta\geq0$. In particular, the author in \cite{CLTh} generalized the results to the Heisenberg groups. It is worth mentioning that Han et al.\cite{HanLu} derived the Stein-Weiss type inequality on the Heisenberg group via the weighted integral operators. In addition, Chen et al.\cite{ChenLLT} considered the reversed Stein-Weiss type inequality. Furthermore, the Euler-Lagrange equations related to the HLS and Stein-Weiss type inequality have also attracted a lot of interest. On one hand, for the Euler-Lagrange equation related to HLS inequality \eqref{HLs}, the authors in \cite{ChenLi} precisely classified the positive solutions to this integral systems via moving plane argument and regularity lifting technique. On the other hand, the authors in \cite{ChenC} studied the property of positive solutions of the Euler-Lagrange system corresponding to Stein-Weiss inequality \eqref{whls}. Furthermore, Bebernes et al.\cite{BLei} studied the asymptotics of positive solutions of the weighted integral system. For recent development and applications of HLS and Stein-Weiss type inequalities, we refer to \cite{refer1,refer2} and the references therein. 

In fact, many people are also interested in the extremal functions and sharp constant of the integral inequalities on upper half space. In \cite{Hang1}, Hang, Wang and Yan firstly established the following integral inequality with harmonic kernel, 
\begin{equation}\label{hang}
	\int_{\mathbb{R}^{n}_{+}}\int_{\partial\mathbb{R}^{n}_{+}}\frac{x_{n}f(y)g(x)}{\left(|x^{\prime}-y|^{2}+x_{n}^{2}\right)^{\frac{n}{2}}}dxdy\leq C(n)\|f\|_{L^{p}\left(\partial\mathbb{R}^{n}_{+}\right)}\|g\|_{L^{q^{\prime}}\left(\mathbb{R}^{n}_{+}\right)},~~ x=\left(x^{\prime},~x_{n}\right)\in\mathbb{R}^{n}_{+},~ y\in\partial\mathbb{R}^{n}_{+},
\end{equation}
where $1<p,q^{\prime}<\infty$, satisfying $\frac{n-1}{n}\frac{1}{p}+\frac{1}{q^{\prime}}=1$. Actually, inequality \eqref{hang} can be
regarded as Carleman’s inequality in the higher dimension, and its implies the sharp isoperimetric inequality \cite{Hang2}. The authors considered the existence of extremal function for inequality \eqref{hang} through the method of symmetrization and the concentration compactness principle, and they also discussed a sequence of qualitative properties of extremal functions including regularity and symmetry. Later, Chen \cite{Chens} generalized the above inequality to the more general inequality \eqref{hang} with poly-harmonic extension. More precisely, Chen obtained the following integral inequality on the upper half space,
\begin{equation}\nonumber
\|P_{\mu}f\|_{L^{\frac{np}{n-1}}(\mathbb{R}^{n}_{+})}\leq C(n,\mu,p)\|f\|_{L^{p}(\partial\mathbb{R}^{n}_{+})},
\end{equation}
for all $1<p\leq \infty$ and $n\geq2$, where
\begin{equation}\nonumber
	P_{\mu}(f)(x):=\int_{\partial\mathbb{R}^{n}_{+}}\frac{x_{n}^{\mu+1-n}}{\left(|x^{\prime}-y|^{2}+x_{n}^{2}\right)^{\frac{\mu}{2}}}f(y)dy
\end{equation}
is the poly-harmonic extension operator. Moreover, the author classified the positive extremals via the rearrangement method in the case $p=\frac{2(n-1)}{2n-2-\mu}$. Furthermore, Dou and Zhu \cite{DouZ1} explored the sharp HLS inequality on the half space. With the help of Riesz's rearrangement technique, they proved the existence of extremals and computed explicitly the sharp constant.  Recently, Gluck \cite{MG} obtained the following sharp inequalities with the a family of kernels on the upper half space 
\begin{equation}\label{gluck}
	\left|	\int_{\mathbb{R}^{n}_{+}}\int_{\partial\mathbb{R}^{n}_{+}}K\left(x^{\prime}-y,x_{n}\right)f(y)g(x)dxdy\right|\leq C(n,\mu,\lambda,p)\|f\|_{L^{p}\left(\partial\mathbb{R}^{n}_{+}\right)}\|g\|_{L^{q^{\prime}}\left(\mathbb{R}^{n}_{+}\right)},
\end{equation}
where $K$ is a family of kernels
\begin{equation}\nonumber
	K(x)=K_{\mu,\lambda}(x)=\frac{x_{n}^{\lam}}{\left(|x^{\prime}|^{2}+x_{n}^{2}\right)^{\frac{\mu}{2}}}.
\end{equation}
It's worth showing that, this kernel $K$ includes the Riesz kernel and the classical Poisson kernel. By a subcritical method, Gulck computed the best constant for a family of HLS type inequalities \eqref{gluck}, and the author gave a precise classification of the related extremals via the method of moving sphere. Liu \cite{Liuzhao} generalized the Hardy-Littlewood-Sobolev inequality with
general kernel in the conformal invariant case for all critical indices. The reverse inequalities on the upper half spaces are also widely studied. In \cite{DHL}, the authors proved the reversed Hardy-Littlewood-Sobolev inequality with extended kernel $K$.
For the doubly weighted case, the Stein-Weiss type inequality on the half space was established by Dou \cite{Dou}, the author also studied the existence of extremal functions. When the kernel function in \eqref{hang} is replaced by the fractional Poisson kernel, Chen, Lu and Tao \cite{ChenStein} investigated the following inequality with the double weights,
\begin{equation}\label{ChenL}
	\int_{\mathbb{R}^{n}_{+}}\int_{\partial\mathbb{R}^{n}_{+}}|y|^{-\alpha}P(x,y,\mu)f(y)g(x)|x|^{-\beta}dxdy\leq C(n,p,q^{\prime},\alpha,\beta,\mu)\|f\|_{L^{p}\left(\partial\mathbb{R}^{n}_{+}\right)}\|g\|_{L^{q^{\prime}}\left(\mathbb{R}^{n}_{+}\right)},
\end{equation}
where $1<p,q^{\prime}<\infty$, $2\leq\mu<n$ satisfying
\begin{equation}\nonumber
\begin{aligned}
\alpha<\frac{n-1}{p^{\prime}}, \beta<\frac{n+q}{q}, \alpha+\beta\geq0,\\
\frac{n-1}{n}\frac{1}{p}+\frac{1}{q^{\prime}}+\frac{\alpha+\beta+2-\mu}{n}=1,
\end{aligned}
\end{equation}
 and $P(x,y,\mu)=\frac{x_{n}}{\left(|x^{\prime}-y|^{2}+x_{n}^{2}\right)^{\frac{n+2-\mu}{2}}}$. By applying the rearrangement approach and Lorentz interpolation inequality, they studied the existence of extremals of the sharp constant for inequality \eqref{ChenL}. Moreover, they explicit classified the extremals of this inequality via the regularity lifting argument and Poho$\check{z}$aev type identity. Especially, if $\alpha=\beta=0$ in \eqref{ChenL}, the optimal inequality with the fractional Poisson kernel was established by Chen et al.\cite{ChenHLS}. Meanwhile, Chen and his collaborators \cite{ChenRe}, Tao \cite{Tao} considered the reversed Stein-Weiss type inequality on the upper half space respectively.

In this paper, we are devoted to study the existence of extremal functions of the Stein-Weiss type inequality with the optimal constant with a general kernel in the half space $\mathbb{R}^{n}_{+}$, that is to generalize the results obtained by Gluck in \cite{MG} to the double weighted case and to classify all positive extremal functions of Euler-Lagrange equations corresponding to this inequality. First of all, we will investigate the Stein-Weiss type inequality with a general kernel. The first main result of this paper is the following integral inequality.
\begin{thm}\label{thm1}
Let $n\geq3$, $1<p, q^{\prime}<\infty$, $\lambda\geq0$, $\frac{n-1}{n}\frac{1}{p}+\frac{1}{q^{\prime}}\geq1$, $0<\mu<n-2\lambda$ and $\frac{\mu-2\lambda}{2n}+\frac{\mu}{2(n-1)}<1$ satisfies
\begin{equation}\label{in1}
\frac{n-1}{n}\frac{1}{p}+\frac{1}{q^{\prime}}+\frac{\alpha+\beta+\mu-\lambda}{n}=\frac{2n-1}{n}
\end{equation}
with $\alpha<\frac{n-1}{p^{\prime}}$, $\beta<\frac{n+q}{q}$ and $\alpha+\beta\geq0$. Then there exists some positive constant $C(n,p,q^{\prime},\alpha,\beta,\lambda,\mu)$ such that for any functions $f\in L^{p}\left(\partial\mathbb{R}^{n}_{+}\right)$ and $g\in L^{q^{\prime}}\left(\mathbb{R}^{n}_{+}\right)$, it holds that
\begin{equation}\label{in2}
\int_{\mathbb{R}^{n}_{+}}\int_{\partial\mathbb{R}^{n}_{+}}|y|^{-\alpha}P_{\lambda}(x,y,\mu)f(y)g(x)|x|^{-\beta}dxdy\leq C(n,p,q^{\prime},\alpha,\beta,\lambda,\mu)\|f\|_{L^{p}\left(\partial\mathbb{R}^{n}_{+}\right)}\|g\|_{L^{q^{\prime}}\left(\mathbb{R}^{n}_{+}\right)},
\end{equation}
where $P_{\lambda}(x,y,\mu)=\frac{x_{n}^{\lambda}}{\left(|x^{\prime}-y|^{2}+x_{n}^{2}\right)^{\frac{\mu}{2}}}$
with $x=(x^{\prime},x_{n})\in \mathbb{R}^{n-1}\times\mathbb{R}^{+}$.
\end{thm}
In fact, by defining the following integral operator with double weights
\begin{equation}\nonumber
V(f)(x):=\int_{\partial\mathbb {R}^{n}_{+}}|y|^{-\alpha}P_{\lambda}(x,y,\mu)f(y)|x|^{-\beta}dy,\ \ x\in\mathbb{R}^{n}_{+},
\end{equation}
and
\begin{equation}\nonumber
	W(g)(y):=\int_{\mathbb{R}^{n}_{+}}|y|^{-\alpha}P_{\lambda}(x,y,\mu)g(x)|x|^{-\beta}dx,\ \ y\in\partial\mathbb{R}^{n}_{+},
\end{equation}
inequality \eqref{in2} is equivalent to the following weighted inequality via a dual argument
\begin{equation}\label{in3}
\|V(f)\|_{L^{q}\left(\mathbb{R}^{n}_{+}\right)}\leq C(n,p,q^{\prime},\alpha,\beta,\lambda,\mu)\|f\|_{L^{p}\left(\partial\mathbb{R}^{n}_{+}\right)}, 
\end{equation}
and
\begin{equation}\label{in31}
 \|W(g)\|_{L^{p^{\prime}}\left(\partial\mathbb{R}^{n}_{+}\right)}\leq C(n,p,q^{\prime},\alpha,\beta,\lambda,\mu)\|g\|_{L^{q^{\prime}}\left(\mathbb{R}^{n}_{+}\right)}
\end{equation}
where $\frac{n-1}{n}\frac{1}{p}+\frac{\alpha+\beta+\mu-\lambda-n+1}{n}=\frac{1}{q}$ and $\frac{1}{p^{\prime}}=\frac{n}{n-1}\frac{1}{q^{\prime}}+\frac{\alpha+\beta+\mu-\lambda-n}{n-1}$.

Next, based on symmetric rearrangement inequality, we shall prove that the existence of extremal functions to inequality \eqref{in2}.
\begin{thm}\label{thm2}
	Let $n\geq3$, $1<p,q^{\prime}<\infty$, $\lambda\geq0$, $\frac{n-1}{n}\frac{1}{p}+\frac{1}{q^{\prime}}\geq1$, $0<\mu<n-2\lambda$ and $\frac{\mu-2\lambda}{2n}+\frac{\mu}{2(n-1)}<1$ satisfying
	\begin{equation}\label{in4}
		\frac{n-1}{n}\frac{1}{p}+\frac{1}{q^{\prime}}+\frac{\alpha+\beta+\mu-\lambda}{n}=\frac{2n-1}{n}
	\end{equation}
	with $\alpha<\frac{n-1}{p^{\prime}}$, $\beta<\frac{n+q}{q}$ and $\alpha,\beta\geq0$. Then there exists some positive functions $f\in L^{p}\left(\partial\mathbb{R}^{n}_{+}\right)$ satisfying $\|f\|_{L^{p}\left(\partial\mathbb{R}^{n}_{+}\right)}=1$ and $\|V(f)\|_{L^{q}\left(\mathbb{R}^{n}_{+}\right)}=C(n,p,q^{\prime},\alpha,\beta,\lambda,\mu)$. Moreover, if $\left(f(y), g(x)\right)$ is a pair of maximizer of inequality \eqref{in2}, then $f(y)$ is radially symmetric and monotone decreasing about the origin, and there exists some positive constant $c_{0}$ such that $g(x)=c_{0}V(f)(x)$.
\end{thm}

Furthermore, we will study the properties of these extremal functions. For this goal, we may maximize the following functional
\begin{equation}\label{pe1}
J(f,g)=\int_{\mathbb{R}^{n}_{+}}\int_{\partial\mathbb{R}^{n}_{+}}|y|^{-\alpha}P_{\lambda}(x,y,\mu)f(y)g(x)|x|^{-\beta}dxdy,
\end{equation}	
under the assumption $\|f\|_{L^{p}({\partial\mathbb{R}^{n}_{+}})}=\|g\|_{L^{q^{\prime}}({\mathbb{R}^{n}_{+}})}=1$. According to the Euler-Lagrange multiplier theorem, we can deduce the following integral system with double weights
\begin{equation}\label{pe2}
	\left\lbrace 
\begin{aligned}	
&J(f,g)g(x)^{q^{\prime}-1}=\int_{\partial\mathbb{R}^{n}_{+}}|y|^{-\alpha}P_{\lambda}(x,y,\mu)f(y)|x|^{-\beta}dy,&x\in\mathbb{R}^{n}_{+},\\
&J(f,g)f(y)^{p-1}=\int_{\mathbb{R}^{n}_{+}}|y|^{-\alpha}P_{\lambda}(x,y,\mu)g(x)|x|^{-\beta}dx, &y\in\partial\mathbb{R}^{n}_{+}.
\end{aligned}
\right.
\end{equation}	
In particular, we assume that $u=c_{1}f^{p-1}$, $v=c_{2}g^{q^{\prime}-1}$, $q_{0}=\frac{1}{q^{\prime}-1}$ and $p_{0}=\frac{1}{p-1}$ with two suitable constants $c_{1}$ and $c_{2}$ in \eqref{pe2}, the system can be rewritten as
\begin{equation}\label{pm1}
\left\lbrace 
\begin{aligned}
u(y)&=\int_{\mathbb{R}^{n}_{+}}|x|^{-\beta}P_{\lambda}(x,y,\mu)v^{q_{0}}(x)|y|^{-\alpha}dx, &y\in \partial\mathbb{R}^{n}_{+},\\
v(x)&=\int_{\partial\mathbb{R}^{n}_{+}}|y|^{-\alpha}P_{\lambda}(x,y,\mu)u^{p_{0}}(y)|x|^{-\beta}dy, &x\in \mathbb{R}^{n}_{+}.
\end{aligned}
\right.
\end{equation}
where $\alpha,\beta\geq0$, $0<\mu<n-2\lambda$ and $p_{0}, q_{0}$ satisfies the identity $\frac{n-1}{n}\frac{1}{p_{0}+1}+\frac{1}{q_{0}+1}=\frac{\alpha+\beta+\mu-\lambda}{n}$.

By applying the regularity lifting arguments, we can derive the regularity results of the positive solutions for system \eqref{pm1} under the integral assumption.
\begin{thm}\label{re1}
Suppose that $0<\alpha,\beta<n$, $0<\mu<n-2\lambda$, $0<p_{0}, q_{0}<\infty$.
Let $\left(u, v\right)\in L^{p_{0}+1}\left(\partial \mathbb{R}^{n}_{+}\right)\times L^{q_{0}+1}\left(\mathbb{R}^{n}_{+}\right)$ be a pair of positive solutions of the integral system \eqref{pm1}. Then $\left(u, v\right)\in L^{r}\left(\partial \mathbb{R}^{n}_{+}\right)\times L^{s}\left(\mathbb{R}^{n}_{+}\right)$ with
\begin{equation}\nonumber
\frac{1}{r}\in\left(\frac{\alpha}{n-1},\frac{\alpha+\mu-\lambda+1}{n-1}\right)\cap\left(\frac{1}{p_{0}+1}-\frac{n}{n-1}\frac{1}{q_{0}+1}+\frac{\beta-1}{n-1},\frac{1}{p_{0}+1}-\frac{n}{n-1}\frac{1}{q_{0}+1}+\frac{\beta+\mu-\lambda}{n}\right)
\end{equation}
and 
\begin{equation}\nonumber
\frac{1}{s}\in\left(\frac{\beta-1}{n},\frac{\beta+\mu-\lambda}{n}\right)\cap\left(\frac{1}{q_{0}+1}-\frac{n-1}{n}\frac{1}{p_{0}+1}+\frac{\alpha}{n},\frac{1}{q_{0}+1}-\frac{n-1}{n}\frac{1}{p_{0}+1}+\frac{\alpha+\mu-\lambda+1}{n-1}\right).
\end{equation}
\end{thm}
Next, we are going to explore the asymptotic behavior of the positive solutions for system \eqref{pm1}. In light of the regularity theorem, we derive the following result. 
\begin{thm}\label{re2}
Suppose that $0<\alpha,\beta<n$, $0<\mu<n-2\lambda$ and $p_{0}, q_{0}>1$.
Let $\left(u, v\right)\in L^{p_{0}+1}\left(\partial \mathbb{R}^{n}_{+}\right)\times L^{q_{0}+1}\left(\mathbb{R}^{n}_{+}\right)$ be a pair of positive solutions of the integral system \eqref{pm1}.
\begin{enumerate}
\item If $\frac{1}{q_{0}}-\frac{\mu+\beta-\lambda}{q_{0}n}>\frac{\beta-1}{n}$, then
\begin{equation}\nonumber
\lim\limits_{|y|\rightarrow0}u(y)|y|^{\alpha}=\int_{\mathbb{R}^{n}_{+}}\frac{x_{n}^{\lambda}v^{q_{0}}(x)}{|x|^{\mu+\beta}}dx.
\end{equation}
\item If $\frac{1}{p_{0}}-\frac{\mu+\alpha-\lambda+1}{p_{0}(n-1)}>\frac{\alpha}{n-1}$, then
\begin{equation}\nonumber
\lim\limits_{|x|\rightarrow0}\frac{v(x)|x|^{\beta}}{x_{n}^{\lambda}}=\int_{\partial\mathbb{R}^{n}_{+}}\frac{u^{p_{0}}(y)}{|y|^{\alpha+\mu}}dy.
\end{equation}
\end{enumerate}
\end{thm}

In the spirit of inequality \eqref{in2} and the integral assumption, we will apply the method of moving plane in integral form in \cite{WCCL,WCCLB} by Chen et al. to study the symmetry property of the solutions.
\begin{thm}\label{thm3}
Assume that $0<\alpha,\beta<n$, $0<\mu<n-2\lambda$, $0<p_{0}, q_{0}<\infty$. If 
$\left(u, v\right)\in L^{p_{0}+1}\left(\partial \mathbb{R}^{n}_{+}\right)\times L^{q_{0}+1}\left(\mathbb{R}^{n}_{+}\right)$ is pair of positive solution of integral system \eqref{pm1}, where $p_{0}$ and $q_{0}$ satisfying $\frac{n-1}{n}\frac{1}{p_{0}+1}+\frac{1}{q_{0}+1}=\frac{\alpha+\beta+\mu-\lambda}{n}$, then $u(y)$ and $v(x)|_{\partial\mathbb{R}^{n}_{+}}$ must be radially symmetric and monotonicity decreasing about some point $y_{0}\in\partial\mathbb{R}^{n}_{+}$.
\end{thm}
Finally, as a special case, our results can be applied to the study of the following integral system with single weight
\begin{equation}\label{yu1}
	\left\lbrace 
	\begin{aligned}
		u(y)&=\int_{\mathbb{R}^{n}_{+}}|x|^{-\beta}P_{\lambda}(x,y,\mu)v^{q_{0}}(x)dx, &y\in \partial\mathbb{R}^{n}_{+},\\
		v(x)&=\int_{\partial\mathbb{R}^{n}_{+}}|y|^{-\alpha}P_{\lambda}(x,y,\mu)u^{p_{0}}(y)dy, &x\in \mathbb{R}^{n}_{+}.
	\end{aligned}
	\right.
\end{equation}

By the Poho$\check{z}$aev identity, we study the necessary condition for the existence of non-trivial solutions for the single weighted  system \eqref{yu1}.
\begin{thm}\label{thm4}
Assume that $0<p_{0}<\infty$, $0<q_{0}<\infty$ and $0<\mu<n-2\lambda$. Let $\left(u,v\right)\in C^{1}(\partial\mathbb{R}_{+}^{n})\times C^{1}(\mathbb{R}_{+}^{n})$ be a pair of non-negative solutions of the integral system \eqref{yu1} with
\begin{equation}\nonumber
\left(u,v\right)\in L^{p_{0}+1}\left(|y|^{-\alpha}dy,\partial \mathbb{R}^{n}_{+}\right)\times L^{q_{0}+1}\left(|x|^{-\beta}dx,\mathbb{R}^{n}_{+}\right).
\end{equation}
Then it must hold that
\begin{equation}\nonumber
\frac{n-1-\alpha}{p_{0}+1}+\frac{n-\beta}{q_{0}+1}=\mu-\lambda.
\end{equation}
\end{thm}
Clearly, according to the above theorem,  we have the following Liouville type result for nonnegative solutions of the single weighted integral system \eqref{yu1}.
\begin{cor}
For $0<\mu<n-2\lambda$, assume that
\begin{equation}\nonumber
	\frac{n-1-\alpha}{p_{0}+1}+\frac{n-\beta}{q_{0}+1}\not=\mu-\lambda,
\end{equation}
then there are no non-trivial solutions $\left(u,v\right)\in L^{p_{0}+1}\left(\partial \mathbb{R}^{n}_{+}\right)\times L^{q_{0}+1}\left(\mathbb{R}^{n}_{+}\right)$ of the integral system \eqref{yu1}.
\end{cor}

The second part of this paper is devoted to the study of the Hartree type elliptic equations on half space, we consider
\begin{equation}\label{main1}
	\left\lbrace 
	\begin{aligned}
		-\Delta u(x)&=\left(\int_{\partial\mathbb{R}^{n}_{+}}\frac{F\left(u(y)\right)}{|x|^{\beta}|x-y|^{\mu}|y|^{\alpha}}dy\right)x_{n}^{\lambda}g\left(u(x)\right),\ \ x\in\mathbb{R}^{n}_{+},	\\	
		\frac{\partial u}{\partial \upsilon}(y)&=\left(\int_{\mathbb{R}^{n}_{+}}\frac{G\left(u(x)\right)x_{n}^{\lambda}}{|x|^{\beta}|x-y|^{\mu}|y|^{\alpha}}dx\right)f\left(u(y)\right),\ \ y\in\partial\mathbb{R}^{n}_{+}.
	\end{aligned}
	\right.
\end{equation}	
where the parameters $\lambda,\alpha,\mu,\beta$ and the functions $G,F,f,g$ will be given specific conditions later. It's well-known that, Hartree type equation is an essential equation in mathematics and physics for the study of the Hartree-Fock model. It is also widely used in Bose–Einstein condensates theory to study how to avoid collapse phenomena. Moreover, the boson star is used to describe the source of dark matter in classical quantum mechanics. Then problem \eqref{main1} plays an important role in development of Black Hole theory. For convenience, the reader may see \cite{ELgartA,LiebYau} and the references therein for more backgrounds about the Hartree type equations.

In term of the Hartree type equations, the qualitative properties of solutions such as symmetry, monotonicity and non-existence have received a great deal of interest in the last years. For the symmetry and monotonicity results, by applying the various versions of the method of moving plane, Lei \cite{YL2}, Du and Yang \cite{DY} finished the classification of positive solutions for Hartree type equation with critical exponent in the whole space. In \cite{ChenY}, the authors established the same non-existence results for Hartree type equation with the boundary conditions on the half space. As an application of inequality \eqref{in2}, we are ready to study the monotonicity and non-existence results of positive solution for problem \eqref{main1} via moving plane argument. In order to present our main result precisely, we firstly give the definition of weak solution to the Hartree type equations \eqref{main1}. 
\begin{Def}
We call that  $u\in W^{1,2}_{loc}(\mathbb{R}^{n}_{+})\cap C^{0}(\overline{\mathbb{R}^{n}_{+}})$ is a weak solution of Hartree type elliptic equations \eqref{main1} if it satisfies for all $\varphi\in C^{\infty}_{c}(\overline{\mathbb{R}^{n}_{+}})$,
\begin{equation}\nonumber
	\begin{aligned}
		\int_{\mathbb{R}^{n}_{+}}\nabla u(x)\nabla \varphi(x) dx&=
		\int_{\mathbb{R}^{n}_{+}}\int_{\partial\mathbb{R}^{n}_{+}}\frac{F(u(y))x_{n}^{\lambda}g(u(x))\varphi(x)}{|x|^{\beta}|x-y|^{\mu}|y|^{\alpha}}dxdy\\
		&+\int_{\mathbb{R}^{n}_{+}}\int_{\partial\mathbb{R}^{n}_{+}}\frac{G(u(x))x_{n}^{\lambda}f(u(y))\varphi(y)}{|x|^{\beta}|x-y|^{\mu}|y|^{\alpha}}dxdy
	\end{aligned}
\end{equation}
on the upper half space.
\end{Def}
Now we are in a position to present our main results.
\begin{thm}\label{app1}
Assume that $0<\alpha, \beta<n$, $0<\lambda<n$ and $0<\mu<n-2\lambda$. Let $u\in W^{1,2}_{loc}(\mathbb{R}^{n}_{+})\cap C^{0}(\overline{\mathbb{R}^{n}_{+}})$ be a positive solution of the system \eqref{main1}. Suppose further the functions $f(t),~g(t),~F(t),~G(t):\left[0,\infty\right)\rightarrow\left[0,\infty\right)$ are continuous in $\left[0,\infty\right)$ satisfying the following conditions,
\begin{enumerate}
\item $f(t)$, $g(t)$, $F(t)$ and $G(t)$ are  increasing in $\left(0,+\infty\right)$,

\item $H(t)=\frac{F(t)}{t^{\frac{2(n-1)-(2\alpha+\mu)}{n-2}}}$, $K(t)=\frac{G(t)}{t^{\frac{2n+2\lambda-(2\beta+\mu)}{n-2}}}$, $k(t)=\frac{g(t)}{t^{\frac{n+2\lambda+2-(2\beta+\mu)}{n-2}}}$ and
$h(t)=\frac{f(t)}{t^{\frac{n-(2\alpha+\mu)}{n-2}}}$ are non-increasing in $\left(0,+\infty\right)$.
\end{enumerate}
Then $u$ depend only on $x_{n}$.
\end{thm}	
It is worth observing that Theorem \ref{app1} provide a useful method to discuss the non-existence of positive solutions for Hartree type equations \eqref{main1} with the special case $\alpha=\beta=0$. The main content is the next result. 
\begin{cor}\label{rem}
Under the assumption of Theorem \ref{app1}, Suppose that at least one of the functions $h$, $k$, $H$ and $K$ is not a constant in $\Big(0, \sup\limits_{y\in\partial\mathbb{R}^{n}_{+}}u(y)\Big)$ or $\Big(0,\sup\limits_{x\in\mathbb{R}^{n}_{+}}u(x)\Big)$. Then $u=\tilde{c}$ with $F(\tilde{c})=G(\tilde{c})=0$.
\end{cor}

The rest of this paper is organized as follows. In Section 2, we mainly consider the sharp Stein-Weiss type inequality with a general kernel on the upper half space. Then we apply the Riesz's rearrangement inequality and Lorentz norm to obtain the existence of extremals for this inequality. In Section 3, by applying the regularity lifting argument and the method of moving plane in integral form, we obtain the qualitative properities of the non-negative solutions to the integral system with the double weights. In Section 4, by using the Poho$\check{z}$aev identiy, we shall show that the necessary condition for the existence of solutionsof the  single weighted integral system. In the last section, we study the weak solutions to Hartree type equation and prove the symmetry and non-existence results.

\section{Stein-Weiss type inequality and sharp constant}
In this section, we will investigate the Stein-Weiss type inequality \eqref{in2} and the existence of extremal functions for the sharp constant $C(\alpha,\beta,\mu,\lambda,p,q^{\prime},n)$. For the sake of simplicity, we firstly introduce some symbols by defining
\begin{equation}\nonumber
	B_{R}(x)=\left\lbrace \xi\in\mathbb{R}^{n}: |\xi-x|<R, x\in\mathbb{R}^{n} \right\rbrace, 
\end{equation}
\begin{equation}\nonumber
	B_{R}^{n-1}(x)=\left\lbrace \xi\in\partial\mathbb{R}^{n}_{+}: |\xi-x|<R, x\in\partial\mathbb{R}^{n}_{n} \right\rbrace, 
\end{equation}
\begin{equation}\nonumber
	B^{+}_{R}(x)=\left\lbrace \xi=(\xi_{1},\xi_{2},...,\xi_{n})\in B_{R}(x): \xi_{n}>0, x\in\mathbb{R}^{n}\right\rbrace. 
\end{equation}
In particular, we write $C$ or $C_{i}$ to denote fifferent non-negative constants, where the value may be different from line to line.
\subsection{Stein-Weiss type inequality} In this subsection, we will use the following integral estimates to prove Theorem \ref{thm1}. The method of these integral estimates on the upper half space was established in \cite{JDou}.
\begin{lem}\label{Lem1}
Assume that $W(x)$ and $U(y)$ are two positive locally integrable functions defined on $\mathbb{R}^{n}_{+}$ and $\partial\mathbb{R}^{n}_{+}$ respectively, for $1<p\leq q<\infty$ and $f$ is non-negative on $\partial\mathbb{R}^{n}_{+}$, then
\begin{equation}\label{lem1}
\left(\int_{\mathbb{R}^{n}_{+}}W(x)\left(\int_{B^{n-1}_{|x|}}f(y)dy\right)^{q}dx\right)^{\frac{1}{q}}\leq C(p,q)\left(\int_{\partial\mathbb{R}^{n}_{+}}f^{p}(y)U(y)dy\right)^{\frac{1}{p}},
\end{equation}
holds if and only if
\begin{equation}\label{lem2}
A_{0}=\sup\limits_{R>0}\left\lbrace \left(\int_{|x|\geq R}W(x)dx\right)^{\frac{1}{q}}\left(\int_{|y|\leq R}U^{1-p^{\prime}}(y)dy\right)^{\frac{1}{p^{\prime}}}\right\rbrace <\infty.
\end{equation}
While,
\begin{equation}\label{lem3}
	\left(\int_{\mathbb{R}^{n}_{+}}W(x)\left(\int_{\partial\mathbb{R}^{n}_{+}\backslash B^{n-1}_{|x|}dy}f(y)\right)^{q}dx\right)^{\frac{1}{q}}\leq C(p,q)\left(\int_{\partial\mathbb{R}^{n}_{+}}f^{p}(y)U(y)dy\right)^{\frac{1}{p}},
\end{equation}
holds if and only if
\begin{equation}\label{lem4}
	A_{1}=\sup\limits_{R>0}\left\lbrace \left(\int_{|x|\leq R}W(x)dx\right)^{\frac{1}{q}}\left(\int_{|y|\geq R}U^{1-p^{\prime}}(y)dy\right)^{\frac{1}{p^{\prime}}}\right\rbrace <\infty.
\end{equation}
\end{lem}
By applying the above Lemma, we are ready to complete the proof of Theorem \ref{thm1}.

\textbf{Proof of Theorem \ref{thm1}.} Without loss of generality, we may suppose that function $f$ is positive. In addition, we consider the double weighted integral operator related to a general kernel as follows
\begin{equation}\nonumber
P_{\lambda}(f)(x)=\int_{\partial\mathbb{R}^{n}_{+}}P_{\lambda}(x,y,\mu)f(y)dy.
\end{equation}
Obviously, we note that inequality \eqref{in2} is equivalent to the following inequality, 
\begin{equation}\nonumber
\|P_{\lambda}(f)|x|^{-\beta}\|_{L^{q}(\mathbb{R}^{n}_{+})}\leq C(n,\alpha,\beta,\lambda,\mu, p,q^{\prime})\|f|y|^{\alpha}\|_{L^{p}(\partial\mathbb{R}^{n}_{+})}.
\end{equation}
Since $q>1$, we may split the integral items into the following three parts, that is 
\begin{equation}\nonumber
\|P_{\lambda}(f)|x|^{-\beta}\|_{L^{q}(\mathbb{R}^{n}_{+})}\lesssim P_{\lambda,1}+P_{\lambda,2}+P_{\lambda,3},
\end{equation}
where
\begin{equation}\nonumber
\begin{aligned}
& P_{\lambda,1}=\int_{\mathbb{R}^{n}_{+}}\left(|x|^{-\beta}\int_{B^{n-1}_{\frac{|x|}{2}}}\frac{x_{n}^{\lambda}f(y)}{\left(|x^{\prime}-y|^{2}+x_{n}^{2}\right)^{\frac{\mu}{2}}}dy\right)^{q}dx,\\
& P_{\lambda,2}=\int_{\mathbb{R}^{n}_{+}}\left(|x|^{-\beta}\int_{\partial\mathbb{R}^{n}_{+}\backslash B^{n-1}_{2|x|}}\frac{x_{n}^{\lambda}f(y)}{\left(|x^{\prime}-y|^{2}+x_{n}^{2}\right)^{\frac{\mu}{2}}}dy\right)^{q}dx,
\end{aligned}
\end{equation}
and
\begin{equation}\nonumber
P_{\lambda,3}=\int_{\mathbb{R}^{n}_{+}}\left(|x|^{-\beta}\int_{B^{n-1}_{2|x|}\backslash B^{n-1}_{\frac{|x|}{2}}}\frac{x_{n}^{\lambda}f(y)}{\left(|x^{\prime}-y|^{2}+x_{n}^{2}\right)^{\frac{\mu}{2}}}dy\right)^{q}dx.
\end{equation}
Based on the above analysis, we only need to prove
\begin{equation}\nonumber
P_{\lambda,i}\leq C(n,\alpha,\beta,\lambda,\mu, p,q^{\prime})\|f|y|^{\alpha}\|_{L^{p}(\partial\mathbb{R}^{n}_{+})}^{q}, \quad i=1,2,3.
\end{equation}

Firstly, we estimate $P_{\lambda,1}$. From the definition of $P_{\lambda,1}$, we get
\begin{equation}\label{prf1}
\begin{aligned}
P_{\lambda,1}&=\int_{\mathbb{R}^{n}_{+}}\left(|x|^{-\beta}\int_{B^{n-1}_{\frac{|x|}{2}}}\frac{x_{n}^{\lambda}f(y)}{\left(|x^{\prime}-y|^{2}+x_{n}^{2}\right)^{\frac{\mu}{2}}}dy\right)^{q}dx\\
&\lesssim\int_{\mathbb{R}^{n}_{+}}|x|^{-\beta q-(\mu-\lambda)q}\left(\int_{B^{n-1}_{\frac{|x|}{2}}}f(y)dy\right)^{q}dx.
\end{aligned}
\end{equation}
Set $W(x)=|x|^{-\beta q-(\mu-\lambda)q}$ and $U(y)=|y|^{\alpha p}$, according to Lemma \ref{Lem1}, to verify
\begin{equation}\nonumber
P_{\lambda,1}\leq C(n,\alpha,\beta,\lambda,\mu, p,q^{\prime})\|f|y|^{\alpha}\|^{q}_{L^{p}(\partial\mathbb{R}^{n}_{+})},
\end{equation}
one only need to show that $W(x)$ and $U(y)$ satisfies \eqref{lem2}. In fact, since  $\alpha<\frac{n-1}{p^{\prime}}$, for any $R>0$, we have
\begin{equation}\label{prf2}
	\begin{aligned}
\int_{|x|\geq R}W(x)dx&=\int_{|x|\geq R}|x|^{-\beta q-(\mu-\lambda)q}dx\\
&=\int_{\partial B_{1}^{+}}dy\int_{R}^{\infty}r^{-\beta q-(\mu-\lambda)q}dr\\
&=C(n,\beta,\lambda,\mu,q)R^{n-\beta q-(\mu-\lambda)q},
	\end{aligned}
\end{equation}
and
\begin{equation}\label{prf3}
	\begin{aligned}
\int_{|y|\leq R}U^{1-p^{\prime}}(y)dy &=\int_{|y|\leq R}\left(|y|^{\alpha p}\right)^{1-p^{\prime}}dy \\
&=\int_{s^{n-2}}d\nu\int_{0}^{R}r^{\alpha p(1-p^{\prime})}dr
&=C(n,\alpha, p)R^{\alpha p(1-p^{\prime})+n-1}.
	\end{aligned}
\end{equation}
It follows from \eqref{prf2}, \eqref{prf3} and \eqref{in2} that
\begin{equation}\nonumber
\begin{aligned}
\left(\int_{|x|\geq R} W(x)dx\right)^{\frac{1}{q}}\left(\int_{|y|\leq R}U^{1-p^{\prime}}(y)dy\right)^{p^{\prime}}&<C(n,\alpha,\beta,\lambda,\mu, p,q^{\prime})R^{-\beta-(\mu-\lambda)+\frac{n}{q}+\frac{\alpha p(1-p^{\prime})+n-1}{p^{\prime}}}\\
&=C(n,\alpha,\beta,\lambda,\mu, p,q^{\prime}).
\end{aligned}
\end{equation}

Next, we consider $P_{\lambda,2}$. Noticing that $|y|\geq 2x$, we know $|y-x|\geq \frac{|x|}{2}$. Setting $W(x)=|x|^{(-\beta+\lambda)q}$ and $U(y)=|y|^{(\mu+\alpha)p}$ in \eqref{lem3}, we know
\begin{equation}\nonumber
P_{\lambda,2}\lesssim \int_{\mathbb{R}^{n}_{+}}|x|^{(-\beta+\lambda)q}\left(\int_{\partial\mathbb{R}^{n}_{+}\backslash B^{n-1}_{2|x|}}f(y)|y|^{-\mu}dy\right)^{q}dx\leq C(n,\alpha,\beta,\lambda,\mu, p,q^{\prime})\|f|y|^{\alpha}\|^{q}_{L^{p}(\partial\mathbb{R}^{n}_{+})}.
\end{equation}
We only need to check that $W(x)$ and $U(y)$ satisfies the condition \eqref{lem4}. Since $\beta<\frac{n}{q}+\lambda$, for $R>0$ we have
\begin{equation}\nonumber
\begin{aligned}
\int_{|x|\geq R}W(x)dx&=\int_{|x|\geq R}|x|^{(-\beta+\lambda)q}dx=\int_{\partial B_{1}^{+}}d\nu\int_{R}^{\infty}r^{(-\beta+\lambda)q}dr\\
&=C(n,\beta,\lambda,q)R^{(-\beta+\lambda)q+n},
\end{aligned}
\end{equation}
and 
\begin{equation}\nonumber
	\begin{aligned}
\int_{|y|\leq R}U^{1-{p^{\prime}}}(y)dy&=\int_{|y|\leq R}|y|^{\left((\mu+\alpha)p\right)1-{p^{\prime}}}dy=\int_{s^{n-2}}d\nu\int_{0}^{R}r^{\left((\mu+\alpha)p\right)1-{p^{\prime}}}dr\\
&=C(n,\alpha\mu, p)R^{(\mu+\alpha)p(1-p^{\prime}+n-1)}.
\end{aligned}
\end{equation}
Combining the above estimates, it's easy to find the condition \eqref{lem4} holds. 

Finally, we estimate $P_{\lambda,3}$. By virtue of $\frac{|x|}{2}<|y|<2|x|$ and $\alpha+\beta\geq0$, it follows that
\begin{equation}\nonumber
|x-y|^{\alpha+\beta}< 3^{\alpha+\beta}|y| ^{\alpha+\beta}\leq 3^{\alpha+\beta}2^{\beta}|x|^{\beta}|y|^{\alpha}.
\end{equation}
Furthermore, we get
\begin{equation}\nonumber
\begin{aligned}
P_{\lambda,3}&=\int_{\mathbb{R}^{n}_{+}}\left(|x|^{-\beta}\int_{B^{n-1}_{2|x|}\backslash B^{n-1}_{\frac{|x|}{2}}}\frac{x_{n}^{\lambda}f(y)}{\left(|x^{\prime}-y|^{2}+x_{n}^{2}\right)^{\frac{\mu}{2}}}dy\right)^{q}dx\\
&\leq \int_{\mathbb{R}^{n}_{+}}\left(\int_{B^{n-1}_{2|x|}\backslash B^{n-1}_{\frac{|x|}{2}}}\frac{x_{n}^{\lambda}f(y)|y|^{\alpha}}{|x-y|^{\mu+\alpha+\beta}}dy\right)^{q}dx\\
&\leq \int_{\mathbb{R}^{n}_{+}}\left(\int_{\partial R^{n}_{+}}\frac{x_{n}^{\lambda}f(y)|y|^{\alpha}}{|x-y|^{\mu+\alpha+\beta}}dy\right)^{q}dx.
\end{aligned}
\end{equation}
Under the assumptions of Theorem \ref{thm1}, we know $0<\mu<n-2\lambda$ and $\frac{\mu-2\lambda}{2n}+\frac{\mu}{2(n-1)}<1$. Together with the results in \cite{MG}, we deduce
\begin{equation}\nonumber
P_{\lambda,3}\leq C(n,\alpha,\beta,\lambda,\mu, p,q^{\prime})\|f|y|^{\alpha}\|^{q}_{L^{p}(\partial\mathbb{R}^{n}_{+})}.
\end{equation}
Therefore, the proof is completed.
$\hfill{} \Box$

\phantom{===}

\subsection{The extremal functions for inequality \eqref{in2}} In this subsection, we will prove the existence of extremal functions for inequality \eqref{in2} which was obtained in the previous subsection. More precisely, we point out that thestudy of the existence of the extremal functions of the sharp constant is related to the following variational problem
\begin{equation}\label{uu1}
C(\alpha,\beta,\mu,\lambda,p,q^{\prime},n) :=\sup\left\{\|V(f)\|_{L^{q}(\mathbb{R}^{n}_{+})}: f\geq0,\ \ \|f\|_{L^{p}(\partial\mathbb{R}^{n}_{+})}\right\}.
\end{equation}
We shall exploit that the existence of maximiziers to the above supreme problem via the Riesz rearrangement and Lorentz norm, which implies that the extremal functions of inequality \eqref{in2} are radially symmetric and decreasing about some point. For a measurable function $f$ on $\partial\mathbb{R}^{n}_{+}$, and we introduce the Lorentz norm with $0<r,s<+\infty$ as follows
\begin{equation}\nonumber
\|f\|_{L^{r,s}(\partial\mathbb{R}^{n}_{+})} :=\left\lbrace 
\begin{aligned}
&\left(\int_{0}^{\infty}\left(\frac{1}{t}^{\frac{1}{r}}f^{*}(t)\right)^{s}\frac{dt}{t}\right)^{\frac{1}{s}},&\mbox{if}\ \ s<\infty,\\
&\ \ \sup_{t>0}t^{\frac{1}{p}}f^{*}(t),&\mbox{if}\ \ s=\infty, 
\end{aligned}
\right.
\end{equation}   
where $f^{*}(t)$ denote the decreasing and radially symmetric rearrangement function to $f$.

Clearly, for $1\leq p\leq\infty$, given some positive functions $f$, $g$ and $h$, then the following Riesz rearrangement inequality holds (See \cite{ELML,HBEL})
\begin{equation}\label{uu2}
I(f,g,h)\leq I(f^{*},g^{*},h^{*}), 
\end{equation}
with $I(f,g,h):=\displaystyle\int_{\mathbb{R}^{n}_{+}}\int_{\partial R^{n}_{+}}f(x)g(x-y)h(y)dxdy$.

In the following, we are ready to investigate the existence of the extremal functions.

\textbf{Proof of Theorem \ref{thm2}.}  Suppose that $\left\{f_{j}\right\}_{j}$ is a maximizing sequence of the maximizing problem \eqref{uu1}, i.e.
\begin{equation}\nonumber
\|f_{j}\|_{L^{p}(\partial\mathbb{R}^{n}_{+})}=1\quad \mbox{and} \quad \lim_{j\rightarrow+\infty}\|V(f_{j})\|_{L^{q}(\mathbb{R}^{n}_{+})}= C(\alpha,\beta,\mu,\lambda,p,q^{\prime},n).
\end{equation}
It follows from \eqref{uu2} that 
\begin{equation}\nonumber
\|f^{*}_{j}\|_{L^{p}(\partial\mathbb{R}^{n}_{+})}=\|f_{j}\|_{L^{p}(\partial\mathbb{R}^{n}_{+})}=1,\ \ \mbox{and}  \lim_{j\rightarrow+\infty}\|V(f_{j})\|_{L^{q}(\mathbb{R}^{n}_{+})}\leq \lim_{j\rightarrow+\infty}\|V(f^{*}_{j})\|_{L^{q}(\mathbb{R}^{n}_{+})},
\end{equation}
since $\alpha, \beta\geq0$. As a consequence, we know that $\left\{f_{j}\right\}_{j}$ is a positive radially non-increasing sequence. Next we take any $f\in L^{p}(\partial\mathbb{R}^{n}_{+})$ and set $f^{\kappa}_{j}=\kappa^{-\frac{n-1}{p}}f(\frac{y}{\kappa})$ with $\kappa>0$. It's suffice to find that
\begin{equation}\nonumber
	\|f^{\kappa}_{j}\|_{L^{p}(\partial\mathbb{R}^{n}_{+})}=\|f_{j}\|_{L^{p}(\partial\mathbb{R}^{n}_{+})}\ \ \mbox{and}\ \  \lim_{j\rightarrow+\infty}\|V(f_{j}^{\kappa})\|_{L^{q}(\mathbb{R}^{n}_{+})}\leq \lim_{j\rightarrow+\infty}\|V(f^{\kappa}_{j})\|_{L^{q}(\mathbb{R}^{n}_{+})},
\end{equation}
which implies that $\left\{f_{j}^{\kappa}\right\}_{j}$ is still a maximizing sequence to problem \eqref{uu1}. Furthermore, we write
\begin{equation}\nonumber
e_{1} :=\left(1,0,...,0\right)\in \mathbb{R}^{n-1},\ \ A_{j} :=\sup_{\kappa>0}f_{j}^{\kappa}(e_{1})=\sup_{\kappa>0}\kappa^{-\frac{n-1}{p}}f_{j}\left(\frac{e_{1}}{\kappa}\right).
\end{equation}
By direct calculation, we get
\begin{equation}\label{niu1}
0\leq f_{j}(y)\leq A_{j}|y|^{-\frac{n-1}{p}}\ \ \mbox{and}\ \ \|f_{j}\|_{L^{p,\infty}(\partial\mathbb{R}^{n}_{+})}\leq w_{n-2}^{\frac{1}{p}}A_{j}.
\end{equation}
Moreover, according to the Marcinkiewicz interpolation \cite{RO,EM1} with \eqref{in3}, we can deduce the following identity
\begin{equation}\nonumber
\|V(f)\|_{L^{q}(\mathbb{R}^{n}_{+})}\leq C(\alpha,\beta,\mu,\lambda,p,q^{\prime},n)\|f\|_{L^{p,q}}(\partial\mathbb{R}^{n}_{+}).
\end{equation}
Thus, by means of the above inequality, we have
\begin{equation}
\begin{aligned}
\|V(f_{j})\|_{L^{q}(\mathbb{R}^{n}_{+})}&\leq C(\alpha,\beta,\mu,\lambda,p,q^{\prime},n)\|f_{j}\|_{L^{p,q}}(\partial\mathbb{R}^{n}_{+})\\
&\leq C(\alpha,\beta,\mu,\lambda,p,q^{\prime},n)\|f_{j}\|_{L^{p,\infty}}^{1-\frac{p}{q}}\|f_{j}\|_{L^{p}}^{\frac{p}{q}}\\
&\leq
C(\alpha,\beta,\mu,\lambda,p,q^{\prime},n)A_{j}^{1-\frac{p}{q}}
\end{aligned}
\end{equation}
which immediately indicates that $A_{j}\geq c_{0}$ for some positive constant $c_{0}$. 

On one hand, by choosing $\kappa_{j}>0$ satisfies $f_{j}^{\kappa_{j}}(e_{1})\geq c_{0}$. We replace $\left\{f_{j}\right\}_{j}$ with $\left\{f_{j}^{\kappa_{j}}\right\}_{j}$, and the latter still recorded as $\left\{f_{j}\right\}_{j}$, then for any $j$, we have $\left\{f_{j}\right\}_{j}\geq c_{0}$. For another things, for any $R>0$, it holds that 
\begin{equation}\nonumber
\begin{aligned}
\omega_{n-1}f_{j}^{p}(R)R^{n-1}&\leq\omega_{n-2}\int_{0}^{R}f_{j}^{p}(r)r^{n-2}dr\leq\omega_{n-2}\int_{0}^{\infty}f_{j}^{p}(r)r^{n-2}dr\\
&=\int_{\partial R^{n}_{+}}f_{j}^{p}(y)dy=1.
\end{aligned}
\end{equation}
Based on the above arguments, we get $$0\leq f_{j}(y)\leq\omega_{n-1}^{-\frac{1}{p}}|y|^{-\frac{n-1}{p}}.$$ 
According to Lieb's results based on the Helly theorem as in \cite{Lieb1}, we infer that there exists a positive, radially non-increasing function $f$ such that
\begin{equation}\nonumber
	f_{j}\rightarrow f,\ \  \mbox{a.e.}\ \ \partial\mathbb{R}^{n}_{+}.
\end{equation}
It's clear that for $|y|\leq 1$ and $\|f\|_{L^{p}(\partial\mathbb{R}_{+}^{n})}\leq1$, we have $f(y)\geq c_{0}$. Furthermore, according to the Brezis-Lieb Lemma \cite{BLT}, it holds that
\begin{equation}\label{zh1}
\begin{aligned}
\lim_{j\rightarrow+\infty}\|f_{j}-f\|^{p}_{L^{p}(\partial\mathbb{R}^{n}_{+})}&=\lim_{j\rightarrow+\infty}\|f_{j}\|^{p}_{L^{p}(\partial\mathbb{R}^{n}_{+})}-\|f\|^{p}_{L^{p}(\partial\mathbb{R}^{n}_{+})}\\
&=1-\|f\|^{p}_{L^{p}(\partial\mathbb{R}^{n}_{+})}.
\end{aligned}
\end{equation}
for some constant $C>0$, we have from \eqref{niu1},
\begin{equation}
V(f_{j})(x)\leq C |x|^{-\beta}\int_{\partial\mathbb{R}^{n}_{+}}\frac{x_{n}^{\lambda}}{|y|^{\alpha}\left(|x^{\prime}-y|+x_{n}\right)^{\frac{\mu}{2}}}\frac{1}{|y|^{-\frac{n-1}{p}}}dy.
\end{equation}
In view of the assumptions of Theorem \ref{thm2}, it's easy to find that the above inequality is finite. Therefore, according to the dominated convergence theorem, for $x\in \mathbb{R}^{n}_{+}$, we deduce that  $\lim\limits_{j\rightarrow+\infty}V(f_{j})(x)=V(f)(x)$. 

By virtue of the Brezis-Lieb Lemma, we obtain
\begin{equation}\label{zh2}
\begin{aligned}
\lim_{j\rightarrow+\infty}\|V(f_{j})\|^{q}_{L^{q}(\mathbb{R}^{n}_{+})}&=\|V(f)\|^{q}_{L^{q}(\mathbb{R}^{n}_{+})}+\lim_{j\rightarrow+\infty}\|V(f_{j})-V(f)\|^{q}_{L^{q}(\mathbb{R}^{n}_{+})}\\
&\leq C(\alpha,\beta,\mu,\lambda,p,q^{\prime},n)^{q}\|f\|^{q}_{L^{p}(\partial\mathbb{R}^{n}_{+})}+C(\alpha,\beta,\mu,\lambda,p,q^{\prime},n)^{q}\lim_{j\rightarrow+\infty}\|f_{j}-f\|^{q}_{L^{p}(\partial\mathbb{R}^{n}_{+})}.
\end{aligned}
\end{equation}
Combining \eqref{zh1} with \eqref{zh2}, lt holds that
\begin{equation}\nonumber
1\leq \|f\|^{q}_{L^{p}(\partial\mathbb{R}^{n}_{+})}+\left(1-\|f\|^{p}_{L^{p}(\partial\mathbb{R}^{n}_{+})}\right)^{\frac{q}{p}}.
\end{equation}
Since $p<q$ and $f\neq0$, we deduce that $\|f\|_{L^{p}(\partial\mathbb{R}^{n}_{+})}=1$. With all the analysis above, it's clear that $f$ is a maximizer to the problem \eqref{uu1}. this proof is completed.  
$\hfill{} \Box$

\section{Qualitative analysis of the positive solutions}
In this section, we study the qualitative properties of positive solutions to the integral equations \eqref{pm1}. More precisely, we are ready to obtain the regularity, asymptotic behaviors and symmetry.
First of all, we introduce some basic definitions. Suppose that $V$ is a topological vector space, and we define two fundamental norms $\|\cdotp\|_{X}$ and $\|\cdotp\|_{Y}$ on $V$,
\begin{equation}\nonumber
\|\cdotp\|_{X}, \|\cdotp\|_{Y}: V\rightarrow [0,\infty].
\end{equation}
Let
\begin{equation}\nonumber
X:=\{v\in V:\|v\|_{X}<\infty\}, Y:=\{v\in V:\|v\|_{Y}<\infty\}. 
\end{equation}
We recall that the operator $T:X\rightarrow Y$
\begin{itemize}
\item is named to be contracting if for any $f,h\in X$, there exists some constant $\varrho\in (0,1)$ such that
\begin{equation}\nonumber
\|T(f)-T(h)\|_{X}\leq \varrho\|f-h\|_{Y}.
\end{equation}
\item is named to be shrinking if for any $h\in X$, there exists some constant $\delta\in(0,1)$ such that
\begin{equation}\nonumber
	\|T(h)\|_{X}\leq \delta\|h\|_{Y}.
\end{equation}
\end{itemize}
Clearly, it is not difficult to show that, a linear shrinking operator must be contracting.
 
Next, we recall the classical regularity lifting lemma (See \cite{CL2}), which will plays an essential role in our discussion.
\begin{lem}\cite{CL2}\label{ABC1}
Let $T$ be a contraction map from $X\rightarrow X$ and $Y\rightarrow Y$, $f\in X$ and there exists a function $g\in X\cap Y$ such that $f=Tf+g$ in $X$. Then $f\in X\cap Y$.
\end{lem}
\textbf{Proof of Theorem \ref{re1}.} For any constant $A>0$, we define
\begin{equation}\nonumber
\begin{aligned}
	&u_{A}(y)=
\begin{cases}
	u(y),\ \ |u(y)|>A~\mbox{or}~|y|>A,\\
	0,~~~~~~\mbox{otherwise},
\end{cases}
&v_{A}(x)=
\begin{cases}
	v(x),\ \ |v(x)|>A~\mbox{or}~|x|>A,\\
	0,~~~~~~\mbox{otherwise},
\end{cases}
\end{aligned}
\end{equation}
$u_{B}(y)=u(y)-u_{A}(y)$ and $v_{B}(x)=v(x)-v_{A}(x)$. Define the linear operator $T_{1}$ as
\begin{equation}\nonumber
T_{1}(h)(y)=\int_{\mathbb{R}^{n}_{+}}|x|^{-\beta}P_{\lambda}(x,y,\mu)v_{A}^{q_{0}-1}(x)h(x)|y|^{-\alpha}dx,\ \ y\in \partial\mathbb{R}^{n}_{+},
\end{equation}
and
\begin{equation}\nonumber
T_{2}(h)(x)=\int_{\partial\mathbb{R}^{n}_{+}}|y|^{-\alpha}P_{\lambda}(x,y,\mu)u_{A}^{p_{0}-1}h(x)|x|^{-\beta}dy,\ \ x\in \mathbb{R}^{n}_{+}.
\end{equation}
Noticing that $\left(u, v\right)\in L^{p_{0}+1}\left(\partial \mathbb{R}^{n}_{+}\right)\times L^{q_{0}+1}\left(\mathbb{R}^{n}_{+}\right)$ is a pair of non-negative solutions of the integral system \eqref{pm1}, we have
\begin{equation}\nonumber
\begin{aligned}
u(y)&=\int_{\mathbb{R}^{n}_{+}}|y|^{-\alpha}P_{\lambda}(x,y,\mu)v^{q_{0}}(x)|x|^{-\alpha}dx\\
&=\int_{\mathbb{R}^{n}_{+}}|y|^{-\alpha}P_{\lambda}(x,y,\mu)\left(v_{A}(x)+v_{B}(x)\right)^{q_{0}-1}v(x)|x|^{-\alpha}dx\\
&=\int_{\mathbb{R}^{n}_{+}}|y|^{-\alpha}P_{\lambda}(x,y,\mu)v_{A}^{q_{0}-1}(x)v(x)|x|^{-\alpha}dx+\int_{\mathbb{R}^{n}_{+}}|y|^{-\alpha}P_{\lambda}(x,y,\mu)v_{B}^{q_{0}}(x)|x|^{-\alpha}dx\\
&:=T_{1}(v)(y)+F(y).
\end{aligned}
\end{equation}
Similarly,
\begin{equation}\nonumber
	\begin{aligned}
v(x)&=\int_{\partial\mathbb{R}^{n}_{+}}|y|^{-\alpha}P_{\lambda}(x,y,\mu)u^{p_{0}}(y)|x|^{-\beta}dy\\
&=\int_{\partial\mathbb{R}^{n}_{+}}|y|^{-\alpha}P_{\lambda}(x,y,\mu)\left(u_{A}(y)+u_{B}(y)\right)^{p_{0}-1}u(y)|x|^{-\beta}dy\\
&=\int_{\partial\mathbb{R}^{n}_{+}}|y|^{-\alpha}P_{\lambda}(x,y,\mu)u_{A}^{p_{0}-1}(y)u(y)|x|^{-\beta}dy+\int_{\partial\mathbb{R}^{n}_{+}}|y|^{-\alpha}P_{\lambda}(x,y,\mu)u_{A}(y)+u_{B}^{p_{0}}(y)|x|^{-\beta}dy\\
&:=T_{2}(u)(x)+G(x),
	\end{aligned}
\end{equation}
where
\begin{equation}\nonumber
F(y)=\int_{\mathbb{R}^{n}_{+}}|y|^{-\alpha}P_{\lambda}(x,y,\mu)v_{B}^{q_{0}}(x)|x|^{-\alpha}dx,\ \ G(x)=\int_{\partial\mathbb{R}^{n}_{+}}|y|^{-\alpha}P_{\lambda}(x,y,\mu)u_{A}(y)+u_{B}^{p_{0}}(y)|x|^{-\beta}dy.
\end{equation}
Furthermore, we may define the operator $T:L^{r}(\partial\mathbb{R}^{n}_{+})\times L^{s}(\mathbb{R}^{n}_{+})$,
\begin{equation}\nonumber
T(h_{1},h_{2})=\left(T_{1}(h_{2}), T_{2}(h_{1})\right)
\end{equation}  
with the norm $\|\left(h_{1},h_{2}\right)\|_{r,s}=\|h_{1}\|_{L^{r}(\partial\mathbb{R}^{n}_{+})}+\|h_{2}\|_{L^{s}(\mathbb{R}^{n}_{+})}$. Obviously, it holds that
\begin{equation}\nonumber
\left(u,v\right)=T(u,v)+(F,G).
\end{equation}

To take full advantage of the regularity lifting argument via contracting map, we set the parameters $r$ and $s$ satisfies 
\begin{equation}\nonumber
\frac{1}{s}+\frac{n-1}{n}\frac{1}{p_{0}+1}=\frac{1}{q_{0}+1}+\frac{n-1}{n}\frac{1}{r}.
\end{equation}
Noticing that under the hypothesis of Theorem \ref{re1}, the existence of parameters $r$ and $s$ can be ensure. With the purpose of obtain the desired result that $\left(u, v\right)\in L^{r}\left(\partial \mathbb{R}^{n}_{+}\right)\times L^{s}\left(\mathbb{R}^{n}_{+}\right)$, we only require to prove that, for $A$ sufficiently large, the following three cases hold.
\begin{enumerate}
\item $T$ is shrinking from $L^{p_{0}+1}\left(\partial \mathbb{R}^{n}_{+}\right)\times L^{q_{0}+1}\left(\mathbb{R}^{n}_{+}\right)$ to $L^{p_{0}+1}\left(\partial \mathbb{R}^{n}_{+}\right)\times L^{q_{0}+1}\left(\mathbb{R}^{n}_{+}\right)$.
\item $T$ is shrinking from $L^{r}\left(\partial \mathbb{R}^{n}_{+}\right)\times L^{s}\left(\mathbb{R}^{n}_{+}\right)$ to $L^{r}\left(\partial \mathbb{R}^{n}_{+}\right)\times L^{s}\left(\mathbb{R}^{n}_{+}\right)$.
\item $\left(F,G\right)\in L^{p_{0}+1}\left(\partial \mathbb{R}^{n}_{+}\right)\times L^{q_{0}+1}\left(\mathbb{R}^{n}_{+}\right)\cap L^{r}\left(\partial \mathbb{R}^{n}_{+}\right)\times L^{s}\left(\mathbb{R}^{n}_{+}\right)$.
\end{enumerate}

Based on the above preparation, we now begin the proof of the case (1). In fact, by applying the weighted integral inequality \eqref{in3} and the H\"{o}lder inequality, for $(h_{1},h_{2})\in L^{p_{0}+1}\left(\partial \mathbb{R}^{n}_{+}\right)\times L^{q_{0}+1}\left(\mathbb{R}^{n}_{+}\right)$, we know
\begin{equation}\nonumber
\begin{aligned}
\|T_{1}(h_{2})\|_{L^{p_{0}+1}\left(\partial \mathbb{R}^{n}_{+}\right)}&\leq C_{1}\|v_{A}^{q_{0}-1}\|_{L^{\frac{q_{0}+1}{q_{0}-1}}\left(\mathbb{R}^{n}_{+}\right)}\|h_{2}\|_{L^{q_{0}+1}\left(\mathbb{R}^{n}_{+}\right)}\\
&\leq C_{1}\|v_{A}\|^{q_{0}-1}_{L^{q_{0}+1}\left(\mathbb{R}^{n}_{+}\right)}\|h_{2}\|_{L^{q_{0}+1}\left(\mathbb{R}^{n}_{+}\right)},
\end{aligned}
\end{equation}
and
\begin{equation}\nonumber
	\begin{aligned}
\|T_{2}(h_{1})\|_{L^{q_{0}+1}\left(\mathbb{R}^{n}_{+}\right)}&\leq C_{2} \|u_{A}^{p_{0}-1}\|_{L^{\frac{p_{0}+1}{p_{0}-1}}\left(\partial\mathbb{R}^{n}_{+}\right)}\|h_{1}\|_{L^{p_{0}+1}\left(\partial\mathbb{R}^{n}_{+}\right)}\\
&\leq C_{2}\|u_{A}\|^{p_{0}-1}_{L^{p_{0}+1}\left(\partial\mathbb{R}^{n}_{+}\right)}\|h_{1}\|_{L^{p_{0}+1}\left(\partial\mathbb{R}^{n}_{+}\right)}
	\end{aligned}
\end{equation}
where constant $C_{1}, C_{2}>0$.

Notice that since the integrability $L^{p_{0}+1}\left(\partial \mathbb{R}^{n}_{+}\right)\times L^{q_{0}+1}\left(\mathbb{R}^{n}_{+}\right)$, for $A$ sufficiently large, we deduce
\begin{equation}\nonumber
\|T\left(h_{1},h_{2}\right)\|_{p_{0}+1,q_{0}+1}=\|T_{1}(h_{2})\|_{p_{0}+1}+\|T_{2}(h_{1})\|_{q_{0}+1}\leq\frac{1}{2}\|\left(h_{1},h_{2}\right)\|_{p_{0}+1,q_{0}+1}
\end{equation}
which immediately implies that $T$ is a shrinking operator from $L^{p_{0}+1}\left(\partial \mathbb{R}^{n}_{+}\right)\times L^{q_{0}+1}\left(\mathbb{R}^{n}_{+}\right)$ to itself.

Next, we will make full use of Stein-Weiss type inequality \eqref{in2} with a general kernel to prove the case (2), which is similar to what we used in the case (1). For convenience, we only verify that $\|T_{2}(h_{1})\|_{ L^{s}\left(\mathbb{R}^{n}_{+}\right)}\leq\frac{1}{2}\|h_{1}\|_{L^{r}\left(\partial\mathbb{R}^{n}_{+}\right)}$, since the other situation $\|T_{1}(h_{2})\|_{L^{r}\left(\partial \mathbb{R}^{n}_{+}\right)}\leq\frac{1}{2}\|h_{2}\|_{ L^{s}\left(\mathbb{R}^{n}_{+}\right)}$ can be proved in the same way.

Actually, there exists positive constant $C$,
\begin{equation}\label{da1}
\begin{aligned}
	\|T_{2}(h_{1})\|_{ L^{s}\left(\mathbb{R}^{n}_{+}\right)}&\leq C\|u_{A}^{p_{0}-1}h_{1}\|_{L^{z}(\partial\mathbb{R}^{n}_{+})}\\
	&\leq C\|u_{A}\|^{p_{0}-1}_{L^{p_{0}+1}\left(\partial \mathbb{R}^{n}_{+}\right)}\|h_{1}\|_{L^{r}\left(\partial \mathbb{R}^{n}_{+}\right)}
\end{aligned}
\end{equation}
In view of $u\in L^{p_{0}+1}\left(\partial \mathbb{R}^{n}_{+}\right)$, by selecting $A$ is sufficiently large in \eqref{da1}, then we have
\begin{equation}\nonumber
\|T_{2}(h_{1})\|_{ L^{s}\left(\mathbb{R}^{n}_{+}\right)}\leq\frac{1}{2}\|h_{1}\|_{L^{r}\left(\partial\mathbb{R}^{n}_{+}\right)},\quad \forall h_{1}\in L^{r}\left(\partial \mathbb{R}^{n}_{+}\right).
\end{equation}
In conclusion, we require the parameters $r$, $s$ and $z$ satisfies
\begin{equation}\nonumber
\frac{1}{z}-\frac{1}{r}=\frac{p_{0}-1}{p_{0}+1},
\end{equation}
and
\begin{equation}\nonumber
\begin{aligned}
\frac{1}{s}&=\frac{\alpha+\beta+\mu-\lambda-n+1}{n}+\frac{n-1}{n}\frac{1}{z}\\
&=\frac{\alpha+\beta+\mu-\lambda-n+1}{n}+\frac{n-1}{n}\left(\frac{p_{0}-1}{p_{0}+1}+\frac{1}{r}\right)\\
&=\frac{n-1}{n}\frac{1}{p_{0}+1}+\frac{1}{q_{0}+1}-\frac{n-1}{n}+\frac{n-1}{n}\left(\frac{p_{0}-1}{p_{0}+1}+\frac{1}{r}\right)\\
&=\frac{n-1}{n}\frac{1}{r}+\frac{1}{q_{0}+1}-\frac{n-1}{n}\frac{1}{p_{0}+1},
\end{aligned}
\end{equation}
where we have used the equation's condition
\begin{equation}\nonumber
\frac{n-1}{n}\frac{1}{p_{0}+1}+\frac{1}{q_{0}+1}=\frac{\alpha+\beta+\mu-\lambda}{n}.
\end{equation}
Moreover, it is suffices to show that $\alpha\leq\frac{n-1}{z^{\prime}}$ and $\beta\leq\frac{n+s}{s}$. 

Based on the above analysis, we conclude that $\|T(h_{1},h_{2})\|_{r,s}\leq\frac{1}{2}\|h_{1},h_{2}\|_{r,s}$, which indicates that $T$ is a shrinking operator from $L^{r}\left(\partial \mathbb{R}^{n}_{+}\right)\times L^{s}\left(\mathbb{R}^{n}_{+}\right)$ to itself.

Finally, we prove the case (3), that is $\left(F,G\right)\in L^{p_{0}+1}\left(\partial \mathbb{R}^{n}_{+}\right)\times L^{q_{0}+1}\left(\mathbb{R}^{n}_{+}\right)\cap L^{r}\left(\partial \mathbb{R}^{n}_{+}\right)\times L^{s}\left(\mathbb{R}^{n}_{+}\right)$. It should be noted that, $u_{B}$ and $v_{B}$ are uniformly bounded by $A$. then we know that this case holds. Thus the proof of Theorem \ref{re1} is completed by using the regularity lifting Lemma \ref{ABC1}.
$\hfill{} \Box$

\subsection{Asymptotic Estimates}
In this subsection, we are going to explore the asymptotic behaviors of the non-negative solutions of the integral system \eqref{pm1}.

\textbf{Proof of Theorem \ref{re2}.} In order to prove that 
\begin{equation}\label{le1}
\lim\limits_{|y|\rightarrow0}u(y)|y|^{\alpha}=\int_{\mathbb{R}^{n}_{+}}\frac{x_{n}^{\lambda}v^{q_{0}}(x)}{|x|^{\mu+\beta}}dx.
\end{equation}
First, we verify that
 $\int_{\mathbb{R}^{n}_{+}}\frac{v^{q_{0}}(x)}{|x|^{\mu+\beta-\lambda}}dx<+\infty$.
In fact, for any $R>0$, we observe that
\begin{equation}\nonumber
\int_{\mathbb{R}^{n}_{+}}\frac{x_{n}^{\lambda}v^{q_{0}}(x)}{|x|^{\mu+\beta}}dx=\int_{B^{+}_{R}}\frac{x_{n}^{\lambda}v^{q_{0}}(x)}{|x|^{\mu+\beta}}dx+\int_{\mathbb{R}^{n}_{+}\backslash B^{+}_{R}}\frac{x_{n}^{\lambda}v^{q_{0}}(x)}{|x|^{\mu+\beta}}dx.
\end{equation}

On one hand, since $u\in L^{p_{0}+1}\left(\partial \mathbb{R}^{n}_{+}\right)$, then there exists $y_{0}\in \partial \mathbb{R}^{n}_{+}$ satisfies $|y_{0}|<\frac{R}{2}$ such that $u(y_{0})<+\infty$. From the system \eqref{pm1},
\begin{equation}\nonumber
\begin{aligned}
\int_{\mathbb{R}^{n}_{+}\backslash B^{+}_{R}}\frac{x_{n}^{\lambda}v^{q_{0}}(x)}{|x|^{\mu+\beta}}dx&\leq C\int_{\mathbb{R}^{n}_{+}\backslash B^{+}_{R}}P_{\lambda}(x,y,\mu)v^{q_{0}}(x)|x|^{-\beta}dx+\int_{B^{+}_{R}}P_{\lambda}(x,y,\mu)v^{q_{0}}(x)|x|^{-\beta}dx\\
&\leq C\int_{\mathbb{R}^{n}_{+}}P_{\lambda}(x,y,\mu)v^{q_{0}}(x)|x|^{-\beta}dx=|y_{0}|^{\alpha}u(y_{0})<+\infty,
\end{aligned}
\end{equation}
since $|x-y|<|x|$.

On the other hand, by using the H\"{o}lder inequality, we get
\begin{equation}\nonumber
\int_{B^{+}_{R}}\frac{x_{n}^{\lambda}v^{q_{0}}(x)}{|x|^{\mu+\beta}}dx\leq C\left(\int_{B^{+}_{R}}\left(\frac{1}{|x|^{\mu+\beta-\lambda}}\right)^{t^{\prime}}dx\right)^{\frac{1}{t^{\prime}}}\left(\int_{B^{+}_{R}}v^{q_{0}t}(x)dx\right)^{\frac{1}{t}}.
\end{equation}
With the aim of $\int_{\mathbb{R}^{n}_{+}}\frac{v^{q_{0}}(x)}{|x|^{\mu+\beta-\lambda}}dx<+\infty$, now we require that $\left(\mu+\beta-\lambda\right)t^{\prime}<n$ and
\begin{equation}\nonumber
\frac{1}{q_{0}t}\in\left(\frac{\beta-1}{n},\frac{\beta+\mu-\lambda}{n}\right)\cap\left(\frac{1}{q_{0}+1}-\frac{n-1}{n}\frac{1}{p_{0}+1}+\frac{\alpha}{n},\frac{1}{q_{0}+1}-\frac{n-1}{n}\frac{1}{p_{0}+1}+\frac{\alpha+\mu-\lambda+1}{n-1}\right).
\end{equation}

Noting that $\left(\mu+\beta-\lambda\right)t^{\prime}<n$, we have
\begin{equation}\nonumber
\frac{\mu+\beta-\lambda}{q_{0}n}<\frac{1}{q_{0}}-\frac{1}{q_{0}t}.
\end{equation}
Therefore, since $q_{0}>1$ and $\frac{1}{p_{0}+1}>\frac{\alpha}{n-1}$, it is easy to find that
\begin{equation}\nonumber
\begin{aligned}
\frac{1}{q_{0}}-\frac{\mu+\beta-\lambda}{q_{0}n}&=\frac{1}{q_{0}}-\frac{1}{q_{0}}\left(\frac{n-1}{n}\frac{1}{p_{0}+1}+\frac{1}{q_{0}+1}-\frac{\alpha}{n}\right)\\
&=\frac{1}{q_{0}+1}-\frac{1}{q_{0}}\left(\frac{n-1}{n}\frac{1}{p_{0}+1}-\frac{\alpha}{n}\right)\\
&>\frac{1}{q_{0}+1}-\frac{n-1}{n}\frac{1}{p_{0}+1}+\frac{\alpha}{n},
\end{aligned}
\end{equation}
here we applied the condition
\begin{equation}\nonumber
\frac{n-1}{n}\frac{1}{p_{0}+1}+\frac{1}{q_{0}+1}=\frac{\alpha+\beta+\mu-\lambda}{n}.
\end{equation}
With the above analysis, we can select a suitable parameter $t$ such that $\left(\mu+\beta-\lambda\right)t^{\prime}<n$ and $\|v\|_{L^{q_{0}t}(\mathbb{R}^{n}_{+})}<+\infty$, which immediately means that $\int_{\mathbb{R}^{n}_{+}}\frac{v^{q_{0}}(x)}{|x|^{\mu+\beta-\lambda}}dx<+\infty$. 

Next, we prove \eqref{le1}. Direct calculation yields that
\begin{equation}
\begin{aligned}
&\left|\int_{\mathbb{R}^{n}_{+}}P_{\lambda}(x,y,\mu)v^{q_{0}}(x)|x|^{-\beta}dx-\int_{\mathbb{R}^{n}_{+}}\frac{x_{n}^{\lambda}v^{q_{0}}(x)}{|x|^{\mu+\beta}}dx\right|\\
&\leq\left|\int_{B^{+}_{\rho}}\left(P_{\lambda}(x,y,\mu)v^{q_{0}}(x)|x|^{-\beta}-\frac{x_{n}^{\lambda}v^{q_{0}}(x)}{|x|^{\mu+\beta}}\right)dx\right|+\left|\int_{\mathbb{R}^{n}_{+}\backslash B^{+}_{\rho}}\left(P_{\lambda}(x,y,\mu)v^{q_{0}}(x)|x|^{-\beta}-\frac{x_{n}^{\lambda}v^{q_{0}}(x)}{|x|^{\mu+\beta}}\right)dx\right|\\
&:=A_{1}+A_{2}.
\end{aligned}
\end{equation}

On one hand, for $A_{1}$, we have
\begin{equation}\label{mi1}
	\begin{aligned}
\int_{B^{+}_{\rho}}P_{\lambda}(x,y,\mu)v^{q_{0}}(x)|x|^{-\beta}dx&\leq \int_{B^{+}_{\rho}(y)}\frac{v^{q_{0}}(x)}{|x-y|^{\mu+\beta-\lambda}}dx+\int_{B^{+}_{\rho}}\frac{v^{q_{0}}(x)}{|x-y|^{\mu+\beta-\lambda}}dx\\
&\leq2\|v^{q_{0}}\|_{L^{t}(\mathbb{R}^{n}_{+})}\|
\frac{1}{|x|^{\mu+\beta-\lambda}}\|_{L^{t^{\prime}}(B^{+}_{\rho})}.
	\end{aligned}                
\end{equation}
Taking the limit in \eqref{mi1}, which leads to $\lim\limits_{\rho\rightarrow0}\lim\limits_{|y|\rightarrow0} A_{1}=0$.

On the other hand, according to the Lebesgue dominated convergence theorem for $A_{2}$, we get
\begin{equation}\label{mi2}
\lim\limits_{|y|\rightarrow0}\int_{\mathbb{R}^{n}_{+}\backslash B^{+}_{\rho}}\left(P_{\lambda}(x,y,\mu)v^{q_{0}}(x)|x|^{-\beta}-\frac{x_{n}^{\lambda}v^{q_{0}}(x)}{|x|^{\mu+\beta}}\right)dx=0.
\end{equation}
With the help of \eqref{mi1} and \eqref{mi2}, we derive
\begin{equation}\nonumber
\begin{aligned}
\lim\limits_{|y|\rightarrow0}&\left|\int_{\mathbb{R}^{n}_{+}}P_{\lambda}(x,y,\mu)v^{q_{0}}(x)|x|^{-\beta}dx-\int_{\mathbb{R}^{n}_{+}}\left(\frac{x_{n}^{\lambda}v^{q_{0}}(x)}{|x|^{\mu+\beta}}\right)dx\right|\\
&=\lim\limits_{\rho\rightarrow0}\lim\limits_{|y|\rightarrow0}A_{1}+\lim\limits_{\rho\rightarrow0}\lim\limits_{|y|\rightarrow0}A_{2}=0.
\end{aligned}
\end{equation}
For the other case, if $\frac{1}{p_{0}}-\frac{\mu+\alpha-\lambda+1}{p_{0}(n-1)}>\frac{\alpha}{n-1}$, then
\begin{equation}\nonumber
	\lim\limits_{|x|\rightarrow0}\frac{v(x)|x|^{\beta}}{x_{n}^{\lambda}}=\int_{\partial\mathbb{R}^{n}_{+}}\frac{u^{p_{0}}(y)}{|y|^{\alpha+\mu}}dy.
\end{equation}

Similarly, we can prove the second conclusion. In conclusion, the proof of Theorem \ref{re2} is completed.
$\hfill{} \Box$

\subsection{Symmetry via the moving plane argument}
In this  subsection, we will prove the symmetry of positive solutions under the integral conditions. In order to apply the moving plane argument, we give some basic notations. For any $\tau\in\mathbb{R}$, one write
\begin{equation}\nonumber
	y^{\tau}= \left(2\tau-y_{1},...,y_{n-1}\right),\ \ x^{\tau}=\left(2\tau-x_{1},...,x_{n}\right),\ \
	u(y^{\tau})=u_{\tau}(y),\ \ v(x^{\tau})=v_{\tau}(x),
\end{equation}
and 
\begin{equation}\nonumber
	T_{\tau}=\left\lbrace x\in\mathbb{R}^{n}: x_{1}=\tau\right\rbrace,\ \ \Sigma_{y,\tau}=\left\lbrace y\in\partial\mathbb{R}^{n}_{+}: y_{1}<\tau\right\rbrace,\ \ \ \ \Sigma_{x,\tau}=\left\lbrace x\in\mathbb{R}^{n}_{+}: x_{1}<\tau\right\rbrace.
\end{equation}	

Next, we will show the following equalities, which are useful in the method of moving plane.
\begin{lem}
Supposed that $\left(u,v\right)$ be a pair of non-negative solution of the integral system \eqref{pm1}, for any $y\in\partial\mathbb{R}^{n}_{+}$ and $x\in\mathbb{R}^{n}_{+}$, it holds that
	\begin{equation}\label{gg1}
		\begin{aligned}
			u(y)-u_{\tau}(y)&=\int_{\Sigma_{x,\tau}}P_{\lambda}(x,y,\mu)\left(|y|^{-\alpha}v^{q_{0}}(x)|x|^{-\beta}-|y^{\tau}|^{-\alpha}v_{\tau}^{q_{0}}(x)|x^{\tau}|^{-\beta}\right)dx\\ &+\int_{\Sigma_{x,\tau}}P_{\lambda}(x^{\tau},y,\mu)\left(|y|^{-\alpha}v_{\tau}^{q_{0}}(x)|x|^{-\beta}-|y^{\tau}|^{-\alpha}v^{q_{0}}(x)|x|^{-\beta}\right)dx,
		\end{aligned}
	\end{equation}
	and
	\begin{equation}\label{gg2}
		\begin{aligned}
			v(x)-v_{\tau}(x)&=\int_{\Sigma_{y,\tau}}P_{\lambda}(x,y,\mu)\left(|y|^{-\alpha}u^{p_{0}}(y)|x|^{-\beta}-|y^{\tau}|^{-\alpha}u_{\tau}^{p_{0}}(y)|x^{\tau}|^{-\beta}\right)dy\\ &+\int_{\Sigma_{y,\tau}}P_{\lambda}(x^{\tau},y,\mu)\left(|y|^{-\alpha}u_{\tau}^{p_{0}}(y)|x|^{-\beta}-|y^{\tau}|^{-\alpha}u^{p_{0}}(y)|x|^{-\beta}\right)dy,
		\end{aligned}
	\end{equation}
\end{lem}
\begin{proof}
	By direct computation, we know
	\begin{equation}\nonumber
		\begin{aligned}
			u(y)&=\int_{\mathbb{R}^{n}_{+}}|y|^{-\alpha}P_{\lambda}(x,y,\mu)v^{q_{0}}(x)|x|^{-\beta}dx\\
			&=\int_{\Sigma_{x,\tau}}|y|^{-\alpha}P_{\lambda}(x,y,\mu)v^{q_{0}}(x)|x|^{-\beta}dx+\int_{\Sigma_{x,\tau}}|y|^{-\alpha}P_{\lambda}(x^{\tau},y,\mu)v_{\tau}^{q_{0}}(x)|x^{\tau}|^{-\beta}dx,
		\end{aligned}
	\end{equation}
	and
	\begin{equation}\nonumber
		\begin{aligned}
			u_{\tau}(y)&=\int_{\mathbb{R}^{n}_{+}}|y^{\tau}|^{-\alpha}P_{\lambda}(x,y^{\tau},\mu)v^{q_{0}}(x)|x|^{-\beta}dx\\
			&=\int_{\Sigma_{x,\tau}}|y^{\tau}|^{-\alpha}P_{\lambda}(x,y^{\tau},\mu)v^{q_{0}}(x)|x|^{-\beta}dx+\int_{\Sigma_{x,\tau}}|y^{\tau}|^{-\alpha}P_{\lambda}(x^{\tau},y^{\tau},\mu)v_{\tau}^{q_{0}}(x)|x^{\tau}|^{-\beta}dx.
		\end{aligned}
	\end{equation}
	Since $P_{\lambda}(x^{\tau},y^{\tau},\mu)=P_{\lambda}(x,y,\mu)$ and $P_{\lambda}(x,y^{\tau},\mu)=P_{\lambda}(x^{\tau},y,\mu)$, it's not difficult to find that
	\begin{equation}\nonumber
		\begin{aligned}
			u(y)-u_{\tau}(y)&=\int_{\Sigma_{x,\tau}}P_{\lambda}(x,y,\mu)\left(|y|^{-\alpha}v^{q_{0}}(x)|x|^{-\beta}-|y^{\tau}|^{-\alpha}v_{\tau}^{q_{0}}(x)|x^{\tau}|^{-\beta}\right)dx\\ &+\int_{\Sigma_{x,\tau}}P_{\lambda}(x^{\tau},y,\mu)\left(|y|^{-\alpha}v_{\tau}^{q_{0}}(x)|x|^{-\beta}-|y^{\tau}|^{-\alpha}v^{q_{0}}(x)|x|^{-\beta}\right)dx.
		\end{aligned}
	\end{equation}
	Clearly, in the same ways, we get
	\begin{equation}\nonumber
		\begin{aligned}
			v(x)-v_{\tau}(x)&=\int_{\Sigma_{y,\tau}}\left(|y|^{-\alpha}P_{\lambda}(x,y,\mu)|x|^{-\beta}-|y^{\tau}|^{-\alpha}P_{\lambda}(x,y^{\tau},\mu)|x|^{-\beta}\right)u^{p_{0}}(y)dy\\
			&+\int_{\Sigma_{y,\tau}}\left(|y|^{-\alpha}P_{\lambda}(x,y^{\tau},\mu)|x^{\tau}|^{-\beta}-|y^{\tau}|^{-\alpha}P_{\lambda}(x^{\tau},y^{\tau},\mu)|x^{\tau}|^{-\beta}\right)u_{\tau}^{p_{0}}(y)dy\\
			&=\int_{\Sigma_{y,\tau}}P_{\lambda}(x,y,\mu)\left(|y|^{-\alpha}u^{p_{0}}(y)|x|^{-\beta}-|y^{\tau}|^{-\alpha}u_{\tau}^{p_{0}}(y)|x^{\tau}|^{-\beta}\right)dy\\ &+\int_{\Sigma_{y,\tau}}P_{\lambda}(x^{\tau},y,\mu)\left(|y|^{-\alpha}u_{\tau}^{p_{0}}(y)|x|^{-\beta}-|y^{\tau}|^{-\alpha}u^{p_{0}}(y)|x|^{-\beta}\right)dy.
		\end{aligned}
	\end{equation}
This finishes the proof.
\end{proof}
\textbf{Proof of Theorem \ref{thm3}.} We divide the proof into two steps.

\textbf{Step 1.} We claim that for $\tau$ sufficiently negative, there holds
\begin{equation}\label{st1}
	u_{\tau}(y)\geq u(y),\ \ v_{\tau}(x)\geq v(x),\ \ \forall x\in\Sigma_{x,\tau},\ \ y\in\Sigma_{y,\tau}.
\end{equation}
Let
\begin{equation}\nonumber
	\Sigma_{y,\tau}^{u}=\left\lbrace y\in\Sigma_{y,\tau}| u(y)>u_{\tau}(y)\right\rbrace,\ \ \Sigma_{x,\tau}^{v}=\left\lbrace x\in\Sigma_{x,\tau}| v(x)>v_{\tau}(x)\right\rbrace.
\end{equation}
Thus, we will show that for $\tau$ sufficiently negative, the sets $\Sigma_{y,\tau}^{v}$ and $\Sigma_{x,\tau}^{u}$ must have measure zero.

For any $x\in\Sigma_{x,\tau}^{v}$ and $y\in\Sigma_{y,\tau}^{u}$, since $|x^{\tau}|<|x|$, we know from \eqref{gg1}
\begin{equation}\nonumber
	\begin{aligned}
		u(y)-u_{\tau}(y)&\leq\int_{\Sigma_{x,\tau}}|y|^{-\alpha}\left(P_{\lambda}(x,y,\mu)-P_{\lambda}(x^{\tau},y,\mu)\right)\left(v^{q_{0}}(x)|x|^{-\beta}-v_{\tau}^{q_{0}}(x)|x^{\tau}|^{-\beta}\right)dx\\
		&\leq \int_{\Sigma_{x,\tau}^{v}}|y|^{-\alpha}P_{\lambda}\left(v^{q_{0}}(x)-v_{\tau}^{q_{0}}(x)\right)dx.
	\end{aligned}
\end{equation}
Applying Mean Value Theorem, we get 
\begin{equation}\nonumber
	u(y)-u_{\tau}(y)\leq q_{0}\int_{\Sigma_{x,\tau}^{v}}|y|^{-\alpha}P_{\lambda}(x,y,\mu)v^{q_{0}-1}(x)\left(v(x)-v_{\tau}(x)\right)dx.
\end{equation}
Similarly,
\begin{equation}\nonumber
	v(x)-v_{\tau}(x)\leq p_{0}\int_{\Sigma_{y,\tau}^{u}}|y|^{-\alpha}P_{\lambda}(x,y,\mu)u^{p_{0}-1}(y)\left(u(y)-u_{\tau}(y)\right)dy.
\end{equation}	
Using the H\"{o}lder inequality and Theorem \ref{thm1} in the above inequalities, we obtain
\begin{equation}
	\|u-u_{\tau}\|_{L^{p_{0}+1}(\Sigma_{y,\tau}^{u})}\leq C\|v\|^{q_{0}-1}_{L^{q_{0}+1}(\Sigma_{x,\tau})}\|v-v_{\tau}\|_{L^{q_{0}+1}(\Sigma_{x,\tau}^{v})},
\end{equation}
and 
\begin{equation}
	\|v-v_{\tau}\|_{L^{q_{0}+1}(\Sigma_{x,\tau}^{v})}\leq C\|u\|^{p_{0}-1}_{L^{p_{0}+1}(\Sigma_{y,\tau})}\|u-u_{\tau}\|_{L^{p_{0}+1}(\Sigma_{y,\tau}^{u})}.
\end{equation}
In view of $\left(u, v\right)\in L^{p_{0}+1}\left(\partial \mathbb{R}^{n}_{+}\right)\times L^{q_{0}+1}\left(\mathbb{R}^{n}_{+}\right)$, for $\tau$ sufficiently negative, it holds that
\begin{equation}\nonumber
	\|v\|_{L^{q_{0}+1}(\Sigma_{x,\tau})}\leq \frac{1}{2},\ \ \|u\|_{L^{p_{0}+1}(\Sigma_{y,\tau})}\leq \frac{1}{2},
\end{equation}
which indicates that
\begin{equation}\label{f2}
	\|u-u_{\tau}\|_{L^{p_{0}+1}(\Sigma_{y,\tau}^{u})}=0,\ \ \|v-v_{\tau}\|_{L^{q_{0}+1}(\Sigma_{x,\tau}^{v})}=0.
\end{equation}
Therefore, we deduce that $\Sigma_{y,\tau}^{u}$ and $\Sigma_{x,\tau}^{v}$ must be empty set. This implies that \eqref{st1}. 

\textbf{Step 2.} Inequality \eqref{st1} provides a starting point to move the plane $T_{\tau}$. Furthermore, we move the plane $T_{\tau}$ from $-\infty$ to the right as long as \eqref{st1} keeps. one can denote
\begin{equation}\label{st2}
	\tau_{0}=\sup\left\lbrace \tau|u_{\rho}(y)\geq u(y), v_{\rho}(x)\geq v(x), \rho\leq\tau, \forall y\in \Sigma_{y,\tau}, x\in\Sigma_{x,\tau} \right\rbrace. 
\end{equation}
Assume that $\tau_{0}<0$, then we will show that the solution $u$ and $v$ must be symmetric and monotone about the limiting plane, Namely,
\begin{equation}\label{f0}
	u_{\tau_{0}}(y)\equiv u(y),\ \ v_{\tau_{0}}(x)\equiv v(x),\ \ \forall y\in\Sigma_{y,\tau_{0}}, x\in\Sigma_{x,\tau_{0}}.
\end{equation}
Suppose for such a $\tau_{0}$, on $\Sigma_{y,\tau_{0}}$ and $\Sigma_{x,\tau_{0}}$, we get
\begin{equation}\label{f1}
	u(y)\leq u_{\tau_{0}}(y),\ \  v(x)\leq v_{\tau_{0}}(x),\ \ \mbox{but}\ \ u(y)\not=u_{\tau_{0}}(y),\ \ v(x)\not= v_{\tau_{0}}(x).
\end{equation}	

We will show that the plane can be moved further to the right. To be precise, there exists an positive parameter $\varepsilon$ such that for any $x\in\Sigma_{x,\tau}$ and $y\in\Sigma_{y,\tau}$,
\begin{equation}\label{f3}
	u(y)\leq u_{\tau}(y), v(x)\leq v_{\tau}(x),\ \ \forall\tau\in\left[\tau_{0},\tau_{0}+\varepsilon\right),
\end{equation}
which would contradict with the definition of $\tau_{0}$.

Combining with the first step and \eqref{f1}, we know that $u(y)<u_{\tau_{0}}(y)$ and $v(x)<v_{\tau_{0}}(x)$ in the interior of $\Sigma_{y,\tau_{0}}$ and $\Sigma_{x,\tau_{0}}$ respectively. Moreover, we define
\begin{equation}\nonumber
	\Sigma^{\widetilde{v}}_{x,\tau_{0}}=\left\lbrace x\in\Sigma_{x,\tau_{0}}|v(x)\geq v_{\tau}(x)\right\rbrace,\ \ \Sigma^{\widetilde{u}}_{y,\tau_{0}}=\left\lbrace y\in\Sigma_{y,\tau_{0}}|u(y)\geq u_{\tau}(y)\right\rbrace.
\end{equation}
Therefore, it's easy to find that $\Sigma^{\widetilde{v}}_{x,\tau_{0}}$ and $\Sigma^{\widetilde{u}}_{y,\tau_{0}}$ have measure zero, and $\lim_{\tau\rightarrow\tau_{0}}\Sigma^{u}_{y,\tau}\subset\Sigma^{\widetilde{u}}_{y,\tau_{0}}$, $\lim_{\tau\rightarrow\tau_{0}}\Sigma^{v}_{x,\tau}\subset\Sigma^{\widetilde{v}}_{x,\tau_{0}}$. Furthermore, By virtue of $\left(u, v\right)\in L^{p_{0}+1}\left(\partial \mathbb{R}^{n}_{+}\right)\times L^{q_{0}+1}\left(\mathbb{R}^{n}_{+}\right)$, we can select a adequate small $\varepsilon$ such that for any $\tau$ in $\left[\tau_{0},\tau_{0}+\varepsilon\right)$,
\begin{equation}\nonumber
	\|v\|_{L^{q_{0}+1}(\Sigma_{x,\tau})}\leq \frac{1}{2},\ \ \|u\|_{L^{p_{0}+1}(\Sigma_{y,\tau})}\leq \frac{1}{2}.
\end{equation}
In fact, the estimates is similar to \eqref{f2}, it holds that
\begin{equation}\nonumber
	\|u-u_{\tau}\|_{L^{p_{0}+1}(\Sigma_{y,\tau}^{u})}=0,\ \ \|v-v_{\tau}\|_{L^{q_{0}+1}(\Sigma_{x,\tau}^{v})}=0.
\end{equation}
Based on the above discussion, we conclude that $\Sigma^{u}_{y,\tau}$ and $\Sigma^{v}_{x,\tau}$ must be empty set. This proves \eqref{f3} and hence \eqref{f0}.

If $\tau_{0}=0$, then we can make full use of the same ways in the opposite direction. Indeed, we will get two situations. If the plane $T_{\tau}$ stop at some point before the origin, we can deduce that  $u(y)$ and  $v(x)$ must be symmetric and monotone decreasing in the $x_{1}$-direction based on the previous analysis. If it stays at the origin again, we also have the symmetry and monotonicity result with $x_{1}=0$. Since $x_{1}$ direction can be chosen arbitrarily, we deduce that $u(y)$ and $v(x)|_{\partial\mathbb{R}^{n}_{+}}$ must be radially symmetry and monotone decreasing about some point $y_{0}\in\partial\mathbb{R}^{n}_{+}$. The proof is accomplished.
$\hfill{} \Box$

\section{The single weighted integral system}
In this section, we will consider the necessary condition for the existence of non-negative solutions of the integral system \eqref{yu1} with single weight.
\begin{equation}\nonumber
\left\lbrace 
\begin{aligned}
u(y)&=\int_{\mathbb{R}^{n}_{+}}|x|^{-\beta}P_{\lambda}(x,y,\mu)v^{q_{0}}(x)dx, &y\in \partial\mathbb{R}^{n}_{+},\\
v(x)&=\int_{\partial\mathbb{R}^{n}_{+}}|y|^{-\alpha}P_{\lambda}(x,y,\mu)u^{p_{0}}(y)dy, &x\in \mathbb{R}^{n}_{+}.
\end{aligned}
\right.
\end{equation}
\textbf{Proof of Theorem \ref{thm4}.} Firstly, according to the integration by parts, we get
\begin{equation}\nonumber
\begin{aligned}
\int_{B^{n-1}_{R}\backslash B^{n-1}_{\varepsilon}}&|y|^{-\alpha}u^{p_{0}}(y)\left(y\cdot\nabla u(y)\right)dy \\
&=\frac{1}{p_{0}+1}\int_{B^{n-1}_{R}\backslash B^{n-1}_{\varepsilon}}|y|^{-\alpha}y\cdot\nabla\left(u^{p_{0}+1}(y)\right)dy\\
&=\frac{R^{1-\alpha}}{p_{0}+1}\int_{\partial B^{n-1}_{R}}u^{p_{0}+1}(y)d\tau+\frac{\varepsilon^{1-\alpha}}{p_{0}+1}\int_{\partial B^{n-1}_{\varepsilon}}u^{p_{0}+1}(y)d\tau\\
&\ \ -\frac{n-1-\alpha}{p_{0}+1}\int_{B^{n-1}_{R}\backslash B^{n-1}_{\varepsilon}}u^{p_{0}+1}(y)|y|^{-\alpha}dy.
\end{aligned}
\end{equation}
In the same way, we know
\begin{equation}\nonumber
	\begin{aligned}
\int_{B^{+}_{R}\backslash B^{+}_{\varepsilon}}&|x|^{-\beta}v^{q_{0}}(x)\left(x \cdot\nabla v(x)\right)dx\\
&=\frac{1}{q_{0}+1}\int_{B^{+}_{R}\backslash B^{+}_{\varepsilon}}|x|^{-\beta}x\cdot\nabla\left(v^{q_{0}+1}(x)\right)dx \\
&=\frac{R^{1-\beta}}{q_{0}+1}\int_{\partial B^{+}_{R}}v^{q_{0}+1}(y)d\tau+\frac{\varepsilon^{1-\beta}}{q_{0}+1}\int_{\partial B^{+}_{\varepsilon}}v^{q_{0}+1}(x)d\tau\\
&\ \ -\frac{n-\beta}{q_{0}+1}\int_{B^{+}_{R}\backslash B^{+}_{\varepsilon}}v^{q_{0}+1}(x)|x|^{-\beta}dx.
	\end{aligned}
\end{equation}
In particular, under the assumptions of Theorem \ref{thm4}, we know $\left(u,v\right)\in L^{p_{0}+1}\left(|y|^{-\alpha}dy,\partial R^{n}_{+}\right)\times L^{q_{0}+1}\left(|x|^{-\beta}dx,R^{n}_{+}\right)$, then there exists $R_{i}\rightarrow+\infty$ and $\varepsilon_{i}\rightarrow0$ such that
\begin{equation}\nonumber
R^{1-\alpha}_{i}\int_{\partial B^{n-1}_{R_{i}}}u^{p_{0}+1}(y)d\tau\rightarrow0,\quad R^{1-\beta}_{i}\int_{\partial B^{+}_{R_{i}}}v^{q_{0}+1}(x)d\tau\rightarrow0.
\end{equation}
Similarly, we obtain
\begin{equation}\nonumber
\varepsilon^{1-\alpha}_{i}\int_{\partial B^{n-1}_{\varepsilon_{i}}}u^{p_{0}+1}(y)d\tau\rightarrow0,\quad \varepsilon^{1-\beta}_{i}\int_{\partial B^{+}_{\varepsilon_{i}}}v^{q_{0}+1}(x)d\tau\rightarrow0.
\end{equation}	
Based on the above analysis, we further get 
\begin{equation}\label{yin1}
\begin{aligned}
\int_{\partial\mathbb{R}^{n}_{+}}&|y|^{-\alpha}u^{p_{0}}(y)\left(y\cdot\nabla u(y)\right)dy+\int_{\mathbb{R}^{n}_{+}}|x|^{-\beta}v^{q_{0}}(x)\left(x \cdot\nabla v(x)\right)dx\\
&=-\frac{n-1-\alpha}{p_{0}+1}\int_{\partial\mathbb{R}^{n}_{+}}|y|^{-\alpha}u^{p_{0}+1}(y)dy-\frac{n-\beta}{q_{0}+1}\int_{\mathbb{R}^{n}_{+}}|x|^{-\beta}v^{q_{0}+1}(x)dx.
\end{aligned}
\end{equation}
Noticing that the weighted integral equations \eqref{yu1}, and by a direct calculation, we obtain
\begin{equation}\label{yin2}
\begin{aligned}
\nabla u(y)\cdot y&=\frac{du(\rho y)}{d\rho}|_{\rho=1}\\
&=-\mu\int_{\mathbb{R}^{n}_{+}}P_{\lambda}(x,y,\mu)|x-y|^{-2}(y-x)\cdot y |x|^{-\beta}v^{q_{0}}(x)dx,
\end{aligned}
\end{equation}	
and
\begin{equation}\label{yin3}
	\begin{aligned}
\nabla v(x)\cdot x&=\frac{dv(\rho x)}{d\rho}|_{\rho=1}\\
&=-\mu\int_{\partial\mathbb{R}^{n}_{+}}P_{\lambda}(x,y,\mu)|x-y|^{-2}(x-y)\cdot x|y|^{-\alpha}u^{p_{0}}(y)dy\\
&\ \ +\lambda\int_{\partial\mathbb{R}^{n}_{+}}P_{\lambda}(x,y,\mu)|y|^{-\alpha}u^{p_{0}}(y)dy.
	\end{aligned}
\end{equation}	
It follows from \eqref{yin2} and \eqref{yin3} that 
\begin{equation}\nonumber
\begin{aligned}
\int_{\partial\mathbb{R}^{n}_{+}}&|y|^{-\alpha}u^{p_{0}}(y)\left(y\cdot\nabla u(y)\right)dy+\int_{\mathbb{R}^{n}_{+}}|x|^{-\beta}v^{q_{0}}(x)\left(x \cdot\nabla v(x)\right)dx\\
&=-\left(\mu-\lambda\right)\int_{\mathbb{R}^{n}_{+}}\int_{\partial R^{n}_{+}}|y|^{-\alpha}P_{\lambda}(x,y,\mu)u^{p_{0}+1}(y)v^{q_{0}+1}(x)|x|^{-\beta}dy dx\\
&=-\left(\mu-\lambda\right)\int_{\mathbb{R}^{n}_{+}}P_{\lambda}(x,y,\mu)v^{q_{0}+1}(x)|x|^{-\beta}dx\\
&=-\left(\mu-\lambda\right)\int_{\partial\mathbb{R}^{n}_{+}}P_{\lambda}(x,y,\mu)u^{p_{0}+1}(y)|y|^{-\alpha}dy.
\end{aligned}
\end{equation}
Combining \eqref{yin1} and the above identity, we know that $\frac{n-1-\alpha}{p_{0}+1}+\frac{n-\beta}{q_{0}+1}=\mu-\lambda$, which completes the proof.
$\hfill{} \Box$	
	
\section{Application to nonlocal elliptic equation on the upper half space}		
In this section, as an application of Stein-Weiss type inequality \eqref{in2}, we are interested in studying the symmetry and non-existence of positive solutions for the equations \eqref{main1} by the method of moving plane in half space. In fact, for Hartree type equations, due to the fact that the nonlinearities are of nonlocal forms, this fact implies that this method is not able to use directly because of the lack of the maximum principles. However, with the help of the integral inequality, we can continue to establish the symmetry and non-existence of weakly solutions for the equations \eqref{main1} via moving plane arguments. For convenience, we firstly write
\begin{equation}
	m(x)=\int_{\partial\mathbb{R}^{n}_{+}}\frac{F(u(y))}{|x|^{\beta}|x-y|^{\mu}|y|^{\alpha}}dy,\qquad \eta(y)=\int_{\mathbb{R}^{n}_{+}}\frac{G(u(x))x_{n}^{\lambda}}{|x|^{\beta}|x-y|^{\mu}|y|^{\alpha}}dx,
\end{equation}
then the equations \eqref{main1} can be rewritten by the following form,
\begin{equation}
	\left\lbrace 
	\begin{aligned}
		-\Delta u(x)&=x_{n}^{\lambda}m(x)g(u(x)),\ \ x\in\mathbb{R}^{n}_{+},\\
		\frac{\partial u}{\partial \upsilon}(y)&=\eta(y)f(u(y)),\ \ y\in\partial\mathbb{R}^{n}_{+}.
	\end{aligned}
	\right.
\end{equation}

Next, It should be noted that, due to the lack of the decay property of those solutions, we fail to prove symmetry and nonexistence results for the equations \eqref{main1} via moving plane arguments directly. In order to overcome this difficulty, we will introduce the Kelvin transform of centered at a point. Then let us take any point $x_{p}\in\partial\mathbb{R}^{n}_{+}$ and define the Kelvin transform of $u(x)$, $m(x)$ and $\eta(y)$ as follows.
\begin{equation}\nonumber
	\begin{aligned}
		&v(x)=\frac{1}{|x-x_{\rho}|^{n-2}}u\left(\frac{x-x_{\rho}}{|x-x_{\rho}|^{2}}+x_{\rho}\right),\quad \omega(x)=\frac{1}{|x-x_{\rho}|^{2\beta+\mu}}m\left(\frac{x-x_{\rho}}{|x-x_{\rho}|^{2}}+x_{\rho}\right),\\
		&z(y)=\frac{1}{|y-y_{\rho}|^{2\alpha+\mu}}\eta\left(\frac{y-y_{\rho}}{|y-y_{\rho}|^{2}}+y_{\rho}\right).
	\end{aligned}
\end{equation}
Obviously, by the above definition, for $|x-x_{p}|\geq1$, we have
\begin{equation}\nonumber
v(x)\leq \frac{C}{|x-x_{p}|^{n-2}},\ \ \omega(x)\leq \frac{C}{|x-x_{p}|^{2\beta+\mu}},
\end{equation}
and 
\begin{equation}
z(y)\leq\frac{C}{|y-y_{p}|^{2\alpha+\mu}}.
\end{equation}
It's easy to find that $\left(v,\omega,z\right)$ has the singularities at point $x_{p}$. Here and the rest of this paper, without loss of generality, we take $x_{p}=0$. 

In addition, by a straightforward computation, we obtain $v(x),~\omega(x),~z(y)$ satisfies the following elliptic system
\begin{equation}\label{main2}
	\left\lbrace 
	\begin{aligned}
		-\Delta v(x)&=x_{n}^{\lambda}\omega(x)k\left(|x|^{n-2}v(x)\right)v(x)^{\frac{2\lambda+n+2-(2\beta+\mu)}{n-2}},\ \ x\in\mathbb{R}^{n}_{+},\\
		\frac{\partial v}{\partial \upsilon}(y)&=z(y)h\left(|y|^{n-2}v(y)\right)v(y)^{\frac{n-(2\alpha+\mu)}{n-2}},\ \ y\in\partial\mathbb{R}^{n}_{+}\backslash\left\{0\right\},\\
		\omega(x)&=\int_{\partial\mathbb{R}^{n}_{+}}\frac{H(|y|^{n-2}v(y))}{|x|^{\beta}|x-y|^{\mu}|y|^{\alpha}}v(y)^{\frac{2(n-1)-(2\alpha+\mu)}{n-2}}dy,\ \ x\in\mathbb{R}^{n}_{+},\\
		z(y)&=\int_{\mathbb{R}^{n}_{+}}\frac{K(|x|^{n-2}v(x))x_{n}^{\lambda}}{|x|^{\beta}|x-y|^{\mu}|y|^{\alpha}}v(x)^{\frac{2n+2\lambda-(2\beta+\mu)}{n-2}}dx,\ \ y\in\partial\mathbb{R}^{n}_{+}\backslash\left\{0\right\}.		
	\end{aligned}
	\right. 
\end{equation}	
Therefore, we will turn attention to studying the symmetry and monotonicity of $v(x)$. Moreover, our main approach relied on the moving plane arguments. To achieve this goal, we will give some necessary symbols. For $\delta>0$, we define 
\begin{equation}\nonumber
	\Sigma_{\delta}=\left\{x\in\mathbb{R}^{n}_{+}|x_{1}>\delta\right\},\ \ \partial\Sigma_{\delta}=\left\{x\in\partial\mathbb{R}^{n}_{+}|x_{1}>\delta\right\},\ \ T_{\delta}=\left\{x\in\mathbb{R}^{n}_{+}|x_{1}=\delta\right\},
\end{equation}
and we also denote the reflected point and functions relate to the hyperplane $T_{\delta}$ by
\begin{equation}\nonumber
x^{\delta}=\left(2\delta-x_{1},...,x_{n}\right),\ \ v_{\delta}(x)=v(x^{\delta}), p^{\delta}=\left(2\delta,0,...,0\right).
\end{equation}
Moreover, we write
\begin{equation}\nonumber
\Sigma_{\delta}^{v}=\left\{x\in\Sigma_{\delta}|v(x)>v_{\delta}(x)\right\},\ \ \partial\Sigma_{\delta}^{v}=\left\{x\in\partial\Sigma_{\delta}|v(x)>v_{\delta}(x)\right\}.
\end{equation}

In the light of the above preparations, we shall give two basic inequality, which is useful in the later proof.
\begin{lem}\label{la1}
	Assume that $v$ is non-negative weak solution of the equations \eqref{main1}, then we have
	\begin{equation}\label{pen1}
		\omega(x)-\omega_{\delta}(x)\leq\int_{\partial\Sigma_{\delta}^{v}}\frac{H\left(|y|^{n-2}v(y)\right)}{|x|^{\beta}|x-y|^{\mu}|y|^{\alpha}}\left(v(y)^{\frac{2(n-1)-(2\alpha+\mu)}{n-2}}-v_{\delta}(y)^{\frac{2(n-1)-(2\alpha+\mu)}{n-2}}\right)dy,
	\end{equation}
	and
	\begin{equation}\label{pen2}
		z(y)-z_{\delta}(y)\leq\int_{\Sigma_{\delta}^{v}}\frac{K\left(|x|^{n-2}v(x)\right)x_{n}^{\lambda}}{|x|^{\beta}|x-y|^{\mu}|y|^{\alpha}}\left(v(x)^{\frac{2n+2\lambda-(2\beta+\mu)}{n-2}}-v_{\delta}(x)^{\frac{2n+2\lambda-(2\beta+\mu)}{n-2}}\right)dx.
	\end{equation}
\end{lem}
\begin{proof}
	By direct calculation, we get
	\begin{equation}\nonumber
		\omega(x)=\int_{\partial\Sigma_{\delta}^{v}}\frac{H\left(|y|^{n-2}v(y)\right)}{|x|^{\beta}|x-y|^{\mu}|y|^{\alpha}}v(y)^{\frac{2(n-1)-(2\alpha+\mu)}{n-2}}dy+\int_{\partial\Sigma_{\delta}^{v}}\frac{H\left(|y^{\delta}|^{n-2}v_{\delta}(y)\right)}{|x|^{\beta}|x-y^{\delta}|^{\mu}|y^{\delta}|^{\alpha}}v_{\delta}(y)^{\frac{2(n-1)-(2\alpha+\mu)}{n-2}}dy,
	\end{equation}
	and
	\begin{equation}\nonumber
		\omega_{\delta}(x)=\int_{\partial\Sigma_{\delta}^{v}}\frac{H\left(|y|^{n-2}v(y)\right)}{|x^{\delta}|^{\beta}|x^{\delta}-y|^{\mu}|y|^{\alpha}}v(y)^{\frac{2(n-1)-(2\alpha+\mu)}{n-2}}dy+\int_{\partial\Sigma_{\delta}^{v}}\frac{H\left(|y^{\delta}|^{n-2}v_{\delta}(y)\right)}{|x^{\delta}|^{\beta}|x^{\delta}-y^{\delta}|^{\mu}|y^{\delta}|^{\alpha}}v_{\delta}(y)^{\frac{2(n-1)-(2\alpha+\mu)}{n-2}}dy.
	\end{equation}
	Since $x\in\Sigma_{\delta}$ and $y\in\partial\Sigma_{\delta}$, then we get
	\begin{equation}\nonumber
		\begin{aligned}
			\omega(x)-\omega_{\delta}(x)=&\int_{\partial\Sigma_{\delta}}\frac{1}{|x-y|^{\mu}}\left(\frac{H\left(|y|^{n-2}v(y)\right)v(y)^{\frac{2(n-1)-(2\alpha+\mu)}{n-2}}}{|x|^{\beta}|y|^{\alpha}}-\frac{H\left(|y^{\delta}|^{n-2}v_{\delta}(y)\right)v_{\delta}(y)^{\frac{2(n-1)-(2\alpha+\mu)}{n-2}}}{|x^{\delta}|^{\beta}|y^{\delta}|^{\alpha}}\right)dy\\
			&+\int_{\partial\Sigma_{\delta}}\frac{1}{|x^{\delta}-y|^{\mu}}\left(\frac{H\left(|y^{\delta}|^{n-2}v_{\delta}(y)\right)v_{\delta}(y)^{\frac{2(n-1)-(2\alpha+\mu)}{n-2}}}{|x|^{\beta}|y^{\delta}|^{\alpha}}-\frac{H\left(|y|^{n-2}v(y)\right)v(y)^{\frac{2(n-1)-(2\alpha+\mu)}{n-2}}}{|x|^{\beta}|y|^{\alpha}}\right)dy\\
			&\leq\int_{\Sigma_{\delta}}\frac{1}{|x|^{\beta}|y|^{\alpha}}\left(\frac{1}{|x-y|^{\mu}}-\frac{1}{|x^{\delta}-y|^{\mu}}\right)\\
			&\cdot\left(H\left(|y|^{n-2}v(y)\right)v(y)^{\frac{2(n-1)-(2\alpha+\mu)}{n-2}}-H\left(|y^{\delta}|^{n-2}v_{\delta}(y)\right)v_{\delta}(y)^{\frac{2(n-1)-(2\alpha+\mu)}{n-2}}\right)dy
		\end{aligned}
	\end{equation}
	For one things, if $y\in \partial\Sigma_{\delta}^{v}$, together with the monotonicity of $H$, then we know
	\begin{equation}\nonumber
		H\left(|y|^{n-2}v(y)\right)\leq H\left(|y^{\delta}|^{n-2}v_{\delta}(y)\right).
	\end{equation}
	For another things, if $y\in\partial\Sigma_{\delta}\backslash\partial\Sigma_{\delta}^{v}$, then we have
	\begin{equation}\nonumber
		\begin{aligned}
			H\left(|y|^{n-2}v(y)\right)v(y)^{\frac{2(n-1)-(2\alpha+\mu)}{n-2}}&=\frac{F\left(|y|^{n-2}v(y)\right)}{|y|^{2(n-1)-(2\alpha+\mu)}}\leq\frac{F\left(|y|^{n-2}v_{\delta}(y)\right)}{|y|^{2(n-1)-(2\alpha+\mu)}}\\
			&\leq\frac{F\left(|y^{\delta}|^{n-2}v_{\delta}(y)\right)}{\left(|y^{\delta}|^{n-2}v_{\delta}(y)\right)^{\frac{2(n-1)-(2\alpha+\mu)}{n-2}}}v_{\delta}(y)^{\frac{2(n-1)-(2\alpha+\mu)}{n-2}}\\
			&=H\left(|y^{\delta}|^{n-2}v_{\delta}(y)\right)v_{\delta}(y)^{\frac{2(n-1)-(2\alpha+\mu)}{n-2}}.
		\end{aligned}
	\end{equation}
	Combining with the above analysis, we immediately get the identity \eqref{pen1}.
	
	Similarly, from
	\begin{equation}\nonumber
		z(y)=\int_{\Sigma_{\delta}^{v}}\frac{K\left(|x|^{n-2}v(x)\right)x_{n}^{\lambda}}{|x|^{\beta}|x-y|^{\mu}|y|^{\alpha}}v(x)^{\frac{2n+2\lambda-(2\beta+\mu)}{n-2}}dx+\int_{\Sigma_{\delta}^{v}}\frac{K\left(|x^{\delta}|^{n-2}v_{\delta}(x)\right)x_{n}^{\lambda}}{|x^{\delta}|^{\beta}|x^{\delta}-y|^{\mu}|y|^{\alpha}}v_{\delta}(x)^{\frac{2n+2\lambda-(2\beta+\mu)}{n-2}}dx,
	\end{equation}
	and 
	\begin{equation}\nonumber
		z_{\delta}(y)=\int_{\Sigma_{\delta}^{v}}\frac{K\left(|x|^{n-2}v(x)\right)x_{n}^{\lambda}}{|x|^{\beta}|x-y^{\delta}|^{\mu}|y^{\delta}|^{\alpha}}v(x)^{\frac{2n+2\lambda-(2\beta+\mu)}{n-2}}dx+\int_{\Sigma_{\delta}^{v}}\frac{K\left(|x^{\delta}|^{n-2}v_{\delta}(x)\right)x_{n}^{\lambda}}{|x^{\delta}|^{\beta}|x^{\delta}-y^{\delta}|^{\mu}|y^{\delta}|^{\alpha}}v_{\delta}(x)^{\frac{2n+2\lambda-(2\beta+\mu)}{n-2}}dx,
	\end{equation}
	then we arrive at \eqref{pen2}. The proof is completed.
\end{proof}

\begin{lem}\label{la2}
Under the conditions of Theorem \ref{app1}, for any fixed parameter $\delta>0$, then it holds that
	
	\begin{enumerate}
		\item $v(x)\in L^{\frac{2n}{n-2}}(\Sigma_{\delta})\cup
		L^{\infty}(\Sigma_{\delta})$,
		\item $(v-v_{\delta})^{+}\in L^{\frac{2n}{n-2}}(\Sigma_{\delta})\cup
		L^{\infty}(\Sigma_{\delta})$.
	\end{enumerate}
	Moreover, there exists positive constant $C_{\delta}$, which is non-increasing in $\delta$, such that
	\begin{equation}\label{pe3}
		\begin{aligned}
			\int_{\Sigma_{\delta}}&|\nabla\left(v-v_{\delta}\right)^{+}|^{2}dx\\
			&\leq C_{\delta}\left[\|v(x)\|_{L^{\frac{2n}{n-2}}(\Sigma_{\delta}^{v})}^{\frac{2\lambda+n+2-(2\beta+\mu)}{n-2}}\|v(y)\|_{L^{\frac{2(n-1)}{n-2}}(\partial\Sigma_{\delta}^{v})}^{\frac{n-(2\alpha+\mu)}{n-2}}+\|\omega(x)\|_{L^{\frac{2n}{(2\beta+\mu)-2\lambda}}(\Sigma_{\delta}^{v})}\|v(x)\|^{\frac{2\lambda+4-(2\beta+\mu)}{2n}}_{L^{\frac{2n}{n-2}}(\Sigma_{\delta}^{v})}\right.\\
			&~~\left.+\|z(y)\|_{L^{\frac{2(n-1)}{n-2}}(\partial\Sigma_{\delta}^{v})}\|v(y)\|^{\frac{2-(2\alpha+\mu)}{2(n-1)}}_{L^{\frac{2(n-1)}{n-2}}(\partial\Sigma_{\delta}^{v})}\right]\left(\int_{\Sigma_{\delta}^{v}}|\nabla\left(v-v_{\delta}\right)^{+}|^{2}dx\right).
		\end{aligned}
	\end{equation}
\end{lem}	
\begin{proof}
	Actually, because of $\delta>0$, then there exists a parameter $r>0$ such that $\Sigma_{\delta}\subset\mathbb{R}^{n}_{+}\backslash B_{r}^{+}(0)$. Therefore, in light of the definition and decay property of $v(x)$, we have
	\begin{equation}\nonumber
		v(x),~(v-v_{\delta})^{+}\in L^{\frac{2n}{n-2}}(\Sigma_{\delta})\cup
		L^{\infty}(\Sigma_{\delta}),
	\end{equation}
	where we denote $B_{r}^{+}(0)=\left\{x\in\mathbb{R}^{n}_{+}| |x|<r\right\}$.
	
	Next, in order to remove the singularity of $v(x)$, $\omega(x)$ and $z(y)$, we need to introduce a cut-off function $\phi=\phi_{\varepsilon}(x)\in C^{1}(\mathbb{R}^{n}, \left[0,1\right])$ as below
	\begin{equation}\nonumber
		\phi_{\varepsilon}(x)=
		\left\{
		\begin{aligned}
			&1,\ \ 2\varepsilon\leq|x-p^{\delta}|\leq\frac{1}{\varepsilon},\\ 
			&0,\ \ |x-p^{\delta}|<\varepsilon,\ \  |x-p^{\delta}|>\frac{2}{\varepsilon}.
		\end{aligned}
		\right. 
	\end{equation}
	Furthermore, we require that $|\nabla \phi|\leq\frac{2}{\varepsilon}$ for $\varepsilon<|x-p^{\delta}|<2\varepsilon$ and $|\nabla\phi|\leq2\varepsilon$ for $\frac{1}{\varepsilon}<|x-p^{\delta}|<\frac{2}{\varepsilon}$. In addition, we also define two functions $\psi(x)$ and $\varphi(x)$ satisfies $\varphi=\varphi_{\varepsilon}=\phi_{\varepsilon}^{2}\left(v-v_{\delta}\right)^{+}$ and $\psi=\psi_{\varepsilon}=\phi_{\varepsilon}\left(v-v_{\delta}\right)^{+}$ respectively, it's easy to find that
	\begin{equation}\nonumber
		|\nabla\psi|^{2}=\nabla\left(v-v_{\delta}\right)^{+}\nabla\varphi+\left[\left(v-v_{\delta}\right)^{+}\right]^{2}|\nabla\phi|^{2}.
	\end{equation}
	
	Now, based on the above preparation, we deduce from \eqref{main2}
	\begin{equation}\label{peng1}
		\begin{aligned}
			&\int_{\Sigma_{\delta}\cup\left\{2\varepsilon\leq|x-p^{\delta}|\leq\frac{1}{\varepsilon}\right\}}|\nabla\left(v-v_{\delta}\right)^{+}|^{2}dx\\
			&\leq\int_{\Sigma_{\delta}^{v}}|\nabla\psi(x)|^{2}dx\leq\int_{\Sigma_{\delta}}\nabla\left(v(x)-v_{\delta}(x)\right)^{+}\nabla\varphi dx+\int_{\Sigma_{\delta}^{v}}\left[\left(v(x)-v_{\delta}(x)\right)^{+}\right]^{2}|\nabla\phi_{\varepsilon}|^{2}dx\\
			&=\int_{\Sigma_{\delta}^{v}}-\Delta\left(v-v_{\delta}\right)\varphi(x)dx+\int_{\partial\Sigma_{\delta}^{v}}\frac{\partial\left(v-v_{\delta}\right)(y)}{\partial\upsilon}\varphi(y)dy+\int_{\Sigma_{\delta}^{v}}\left[\left(v(x)-v_{\delta}(x)\right)^{+}\right]^{2}|\nabla\phi_{\varepsilon}|^{2}dx\\
			&=\int_{\Sigma_{\delta}^{v}}\left[x_{n}^{\lambda}\omega(x)k\left(|x|^{n-2}v(x)\right)v(x)^{\frac{2\lambda+n+2-(2\beta+\mu)}{n-2}}-x_{n}^{\lambda}\omega(x^{\delta})k\left(|x^{\delta}|^{n-2}v_{\delta}(x)\right)v_{\delta}(x)^{\frac{2\lambda+n+2-(2\beta+\mu)}{n-2}}\right.\\
			&\times\left.\left(v(x)-v_{\delta}(x)\right)^{+}\phi_{\varepsilon}^{2}(x)dx\right]-\int_{\partial\Sigma_{\delta}^{v}}\left[z(y)h\left(|y|^{n-2}v(y)\right)v(y)^{\frac{n-(2\alpha+\mu)}{n-2}}-z(y^{\delta})h\left(|y^{\delta}|^{n-2}v_{\delta}(y)\right)v_{\delta}(y)^{\frac{n-(2\alpha+\mu)}{n-2}}\right.\\
			&\left.\times\left(v(y)-v_{\delta}(y)\right)^{+}\phi_{\varepsilon}^{2}(y)dy\right]+\int_{\Sigma_{\delta}^{v}}\left[\left(v(x)-v_{\delta}(x)\right)^{+}\right]^{2}|\nabla\phi_{\varepsilon}|^{2}dx\\
			&:=I_{1}+I_{2}+I_{3}.
		\end{aligned}
	\end{equation}
	
	Apparently, we are going to consider the above three integrals. For the integrals $I_{1}$, if $x\in\Sigma_{\delta}^{v}$, then we have 
	\begin{equation}\nonumber
		k\left(|x|^{n-2}v(x)\right)\leq k\left(|x^{\delta}|^{n-2}v_{\delta}(x)\right),
	\end{equation}
	where we used the monotonicity of $k$. Meanwhile, we treat the domain $\Sigma_{\delta}^{v}$ as
	\begin{equation}\nonumber
		\Omega_{1}=\left\{x\in\Sigma_{\delta}^{v}|\omega(x)>\omega_{\delta}(x)\right\}
	\end{equation}
	and
	\begin{equation}\nonumber
		\Omega_{2}=\left\{x\in\Sigma_{\delta}^{v}|\omega(x)\leq\omega_{\delta}(x)\right\}.
	\end{equation}
	If $x\in\Omega_{1}$, we know
	\begin{equation}\label{ng1}
		\begin{aligned}
			x_{n}^{\lambda}&\omega(x)k\left(|x|^{n-2}v(x)\right)v(x)^{\frac{2\lambda+n+2-(2\beta+\mu)}{n-2}}-x_{n}^{\lambda}\omega(x^{\delta})k\left(|x^{\delta}|^{n-2}v_{\delta}(x)\right)v_{\delta}(x)^{\frac{2\lambda+n+2-(2\beta+\mu)}{n-2}}\\
			&=\left[\omega(x)-\omega(x^{\delta})\right]k\left(|x|^{n-2}v(x)\right)v(x)^{\frac{2\lambda+n+2-(2\beta+\mu)}{n-2}}\\
			&\quad+x_{n}^{\lambda}\omega(x^{\delta})\left[k\left(|x|^{n-2}v(x)\right)v(x)^{\frac{2\lambda+n+2-(2\beta+\mu)}{n-2}}-k\left(|x^{\delta}|^{n-2}v_{\delta}(x)\right)v_{\delta}(x)^{\frac{2\lambda+n+2-(2\beta+\mu)}{n-2}}\right]\\
			&\leq x_{n}^{\lambda}\left[\omega(x)-\omega(x^{\delta})\right]k\left(|x|^{n-2}v(x)\right)v(x)^{\frac{2\lambda+n+2-(2\beta+\mu)}{n-2}}\\
			&\quad+x_{n}^{\lambda}\omega(x)k\left(|x|^{n-2}v(x)\right)\left[v(x)^{\frac{2\lambda+n+2-(2\beta+\mu)}{n-2}}-v_{\delta}(x)^{\frac{2\lambda+n+2-(2\beta+\mu)}{n-2}}\right].
		\end{aligned}
	\end{equation}
	While, if $x\in\Omega_{2}$, then we have
	\begin{equation}\label{ng2}
		\begin{aligned}
			x_{n}^{\lambda}&\omega(x)k\left(|x|^{n-2}v(x)\right)v(x)^{\frac{2\lambda+n+2-(2\beta+\mu)}{n-2}}-x_{n}^{\lambda}\omega(x^{\delta})k\left(|x^{\delta}|^{n-2}v_{\delta}(x)\right)v_{\delta}(x)^{\frac{2\lambda+n+2-(2\beta+\mu)}{n-2}}\\
			&\leq x_{n}^{\lambda}\omega(x)\left[k\left(|x|^{n-2}v(x)\right)v(x)^{\frac{2\lambda+n+2-(2\beta+\mu)}{n-2}}-k\left(|x^{\delta}|^{n-2}v_{\delta}(x)\right)v_{\delta}(x)^{\frac{2\lambda+n+2-(2\beta+\mu)}{n-2}}\right]\\
			&\leq x_{n}^{\lambda}\omega(x)k\left(|x|^{n-2}v(x)\right)\left[v(x)^{\frac{2\lambda+n+2-(2\beta+\mu)}{n-2}}-v_{\delta}(x)^{\frac{2\lambda+n+2-(2\beta+\mu)}{n-2}}\right].
		\end{aligned}
	\end{equation}
	
	Clearly, the integral $I_{2}$ can be estimate as we did in the previous calculation. Therefore, for $y\in\partial\Sigma_{\delta}^{v}$, it follow from the monotonicity of $h$ that,
	\begin{equation}\nonumber
		h\left(|y|^{n-2}v(y)\right)\leq h\left(|y^{\delta}|^{n-2}v_{\delta}(y)\right).
	\end{equation}
	In a similar way, we can divide the domain $\partial\Sigma_{\delta}^{v}$ into 
	\begin{equation}\nonumber
		\Omega_{3}=\left\{y\in\partial\Sigma_{\delta}^{v}|z(y)>z_{\delta}(y)\right\}
	\end{equation}
	and
	\begin{equation}\nonumber
		\Omega_{4}=\left\{y\in\partial\Sigma_{\delta}^{v}|z(y)\leq z_{\delta}(y)\right\}.
	\end{equation}
If $y\in\Omega_{3}$, then we have
	\begin{equation}\label{ng3}
		\begin{aligned}
			&z(y)h\left(|y|^{n-2}v(y)\right)v(y)^{\frac{n-(2\alpha+\mu)}{n-2}}-z(y^{\delta})h\left(|y^{\delta}|^{n-2}v_{\delta}(y)\right)v_{\delta}(y)^{\frac{n-(2\alpha+\mu)}{n-2}}\\
			&\leq z(y)h\left(|y|^{n-2}v(y)\right)\left[v(y)^{\frac{n-(2\alpha+\mu)}{n-2}}-v_{\delta}(y)^{\frac{n-(2\alpha+\mu)}{n-2}}\right]+h\left(|y|^{n-2}v(y)\right)v(y)^{\frac{n-(2\alpha+\mu)}{n-2}}\left[z(y)-z(y^{\delta})\right],
		\end{aligned}
	\end{equation}
	it also holds that for $y\in\Omega_{4}$,
	\begin{equation}\label{ng4}
		\begin{aligned}
			z(y)&h\left(|y|^{n-2}v(y)\right)v(y)^{\frac{n-(2\alpha+\mu)}{n-2}}-z(y^{\delta})h\left(|y^{\delta}|^{n-2}v_{\delta}(y)\right)v_{\delta}(y)^{\frac{n-(2\alpha+\mu)}{n-2}}\\
			&\leq z(y)h\left(|y|^{n-2}v(y)\right)\left[v(y)^{\frac{n-(2\alpha+\mu)}{n-2}}-v_{\delta}(y)^{\frac{n-(2\alpha+\mu)}{n-2}}\right].
		\end{aligned}
	\end{equation}
	
	Consequently, inserting \eqref{ng1}, \eqref{ng2},  \eqref{ng3} and \eqref{ng4} into \eqref{peng1}, we must have
	\begin{equation}\label{peng2}
		\begin{aligned}
			&\int_{\Sigma_{\delta}\cup\left\{2\varepsilon\leq|x-p^{\delta}|\leq\frac{1}{\varepsilon}\right\}}|\nabla\left(v-v_{\delta}\right)^{+}|^{2}dx\\
			&\leq I_{3}+C\int_{\Sigma_{\delta}^{v}}x_{n}^{\lambda}\left[\omega(x)-\omega(x^{\delta})\right]k\left(|x|^{n-2}v(x)\right)v(x)^{\frac{2\lambda+n+2-(2\beta+\mu)}{n-2}}\left(v(x)-v_{\delta}(x)\right)^{+}\phi_{\varepsilon}^{2}(x)dx\\
			&+C\int_{\Sigma_{\delta}^{v}}x_{n}^{\lambda}\omega(x)k\left(|x|^{n-2}v(x)\right)v(x)^{\frac{2\lambda+4-(2\beta+\mu)}{n-2}}\left[\left(v(x)-v_{\delta}(x)\right)^{+}\right]^{2}\phi_{\varepsilon}^{2}(x)dx\\
			&+C\int_{\partial\Sigma_{\delta}^{v}}z(y)h\left(|y|^{n-2}v(y)\right)v(y)^{\frac{2-(2\alpha+\mu)}{n-2}}\left[\left(v(y)-v_{\delta}(y)\right)^{+}\right]^{2}\phi_{\varepsilon}^{2}(y)dy\\
			&+C\int_{\partial\Sigma_{\delta}^{v}}h\left(|y|^{n-2}v(y)\right)v(y)^{\frac{n-(2\alpha+\mu)}{n-2}}\left[z(y)-z(y^{\delta})\right]\left(v(y)-v_{\delta}(y)\right)^{+}\phi_{\varepsilon}^{2}(y)dy\\
			&=I_{3}+A_{1}+A_{2}+A_{3}+A_{4}.
		\end{aligned}
	\end{equation}
	
	In fact, based on the above discussion, we focus on estimating the five integrals in the remaining of proof.
	
	Firstly, to estimate the integrals $I_{3}$ accurately, we write $M_{1}=\left\{x\in\Sigma_{\delta}|\varepsilon<|x-p^{\delta}|<2\varepsilon\right\}$ and $M_{2}=\left\{x\in\Sigma_{\delta}|\frac{1}{\varepsilon}<|x-p^{\delta}|<\frac{2}{\varepsilon}\right\}$, then we have
	\begin{equation}\nonumber
		\int_{M_{1}}|\nabla \phi|^{n}dx\leq C\frac{1}{\varepsilon^{n}}\cdot\varepsilon^{n}=C.
	\end{equation}
	Similarly, 
	\begin{equation}\nonumber
		\int_{M_{2}}|\nabla \phi|^{n}dx\leq C\frac{1}{\varepsilon^{n}}\cdot\varepsilon^{n}=C.
	\end{equation}
	As a consequence, when $\varepsilon\rightarrow0$, by the H\"{o}lder inequality and $(v-v_{\delta})^{+}\in L^{\frac{2n}{n-2}}(\Sigma_{\delta})$,  we conclude
	\begin{equation}\nonumber
		I_{3}\leq\left(\int_{M_{1}\subset M_{2}}\left[(v-v_{\delta})^{+}\right]^{\frac{2n}{n-2}}dx\right)^{\frac{n-2}{n}}\left(\int_{\Sigma_{\delta}}|\nabla\phi|^{n}dx\right)^{\frac{2}{n}}\rightarrow0.
	\end{equation}
	
	Secondly, for the integral $A_{1}$, according to Theorem \ref{thm1}, we infer from \eqref{pe1} and the H\"{o}lder inequality that
	\begin{equation}\label{yi1}
		\begin{aligned}
			A_{1}&\leq C_{\delta}\int_{\Sigma_{\delta}^{v}}x_{n}^{\lambda}\left[\omega(x)-\omega(x^{\delta})\right]v(x)^{\frac{2\lambda+n+2-(2\beta+\mu)}{n-2}}\left(v(x)-v_{\delta}(x)\right)^{+}dx\\
			&\leq C_{\delta}\int_{\Sigma_{\delta}^{v}}\int_{\partial\Sigma_{\delta}^{v}}\frac{x_{n}^{\lambda}v(x)^{\frac{2\lambda+n+2-(2\beta+\mu)}{n-2}}\left(v(x)-v_{\delta}(x)\right)^{+}\left(v(y)^{\frac{2(n-1)-(2\alpha+\mu)}{n-2}}-v_{\delta}(y)^{\frac{2(n-1)-(2\alpha+\mu)}{n-2}}\right)}{|x|^{\beta}|x-y|^{\mu}|y|^{\alpha}}dxdy\\
			&\leq C_{\delta}\|v(x)^{\frac{2\lambda+n+2-(2\beta+\mu)}{n-2}}\left(v(x)-v_{\delta}(x)\right)^{+}\|_{L^{\frac{2n}{2n+2\lambda-(2\beta+\mu)}}(\Sigma_{\delta}^{v})}\|v(y)^{\frac{n-(2\alpha+\mu)}{n-2}}\left(v(y)-v_{\delta}(y)\right)^{+}\|_{L^{\frac{2(n-1)}{2n-2-(2\alpha+\mu)}}(\partial\Sigma_{\delta}^{v})}\\
			&\leq C_{\delta}\|v(x)\|_{L^{\frac{2n}{n-2}}(\Sigma_{\delta}^{v})}^{\frac{2\lambda+n+2-(2\beta+\mu)}{n-2}}\|\left(v(x)-v_{\delta}(x)\right)\|_{L^{\frac{2n}{n-2}}(\Sigma_{\delta}^{v})}\|v(y)\|_{L^{\frac{2(n-1)}{n-2}}(\partial\Sigma_{\delta}^{v})}^{\frac{n-(2\alpha+\mu)}{n-2}}\|\left(v(y)-v_{\delta}(y)\right)\|_{L^{\frac{2(n-1)}{n-2}}(\partial\Sigma_{\delta}^{v})}.
		\end{aligned}
	\end{equation}

	Next, for the integral $A_{2}$, combining the H\"{o}lder inequality with the decay estimate of $v$, there exists a positive constant $C_{\delta}$, which is non-increasing in $\delta$, such that
	\begin{equation}\label{yi2}
		\begin{aligned}
			A_{2}&\leq \int_{\Sigma_{\delta}^{v}}x_{n}^{\lambda}\left[\omega(x)-\omega(x^{\delta})\right]k\left(|x|^{n-2}v(x)\right)v(x)^{\frac{2\lambda+n+2-(2\beta+\mu)}{n-2}}\left(v(x)-v_{\delta}(x)\right)^{+}dx\\
			&\leq C_{\delta}\left(\int_{\Sigma_{\delta}^{v}}\omega(x)^{\frac{2n}{(2\beta+\mu)-2\lambda}}dx\right)^{\frac{(2\beta+\mu)-2\lambda}{2n}}\left(\int_{\Sigma_{\delta}^{v}}v(x)^{\frac{2n}{n-2}}dx\right)^{\frac{2\lambda+4-(2\beta+\mu)}{2n}}\left(\int_{\Sigma_{\delta}^{v}}\left[\left(v(x)-v_{\delta}(x)\right)^{+}\right]^{\frac{2n}{n-2}}dx\right)^{\frac{n-2}{n}},
		\end{aligned}
	\end{equation}
	where we used the monotonicity of $k$.
	
	Furthermore, for the integral $A_{3}$, in view of the H\"{o}lder inequality, we have
	\begin{equation}\label{yi3}
		\begin{aligned}
			A_{3}&\leq \int_{\partial\Sigma_{\delta}^{v}}z(y)h\left(|y|^{n-2}v(y)\right)v(y)^{\frac{2-(2\alpha+\mu)}{n-2}}\left[\left(v(y)-v_{\delta}(y)\right)^{+}\right]^{2}dy\\
			&\leq C_{\delta}\left(\int_{\partial\Sigma_{\delta}^{v}}v(y)^{\frac{2(n-1)}{n-2}}dy\right)^{\frac{2-(2\alpha+\mu)}{2(n-1)}}\left(\int_{\partial\Sigma_{\delta}^{v}}z(y)^{\frac{2(n-1)}{2\alpha+\mu}}dy\right)^{\frac{2\alpha+\mu}{2(n-1)}}\left(\int_{\partial\Sigma_{\delta}^{v}}\left[\left(v(y)-v_{\delta}(y)\right)^{+}\right]^{\frac{2(n-1)}{n-2}}dy\right)^{\frac{n-2}{n-1}}.
		\end{aligned}
	\end{equation}
	
	Finally, we estimate the integral $A_{4}$. According to Theorem \ref{thm1}, it follow from \eqref{pe2} and the H\"{o}lder inequality that
	\begin{equation}\label{yi4}
		\begin{aligned}
			A_{4}&\leq C_{\delta}\int_{\partial\Sigma_{\delta}^{v}}v(y)^{\frac{n-(2\alpha+\mu)}{n-2}}\left[z(y)-z(y^{\delta})\right]\left(v(y)-v_{\delta}(y)\right)^{+}dy\\
			&\leq C_{\delta}\int_{\Sigma_{\delta}^{v}}\int_{\partial\Sigma_{\delta}^{v}}\frac{x_{n}^{\lambda}v(y)^{\frac{n-(2\alpha+\mu)}{n-2}}\left(v(y)-v_{\delta}(y)\right)^{+}v(x)^{\frac{n+2+2\lambda-(2\beta+\mu)}{n-2}}\left(v(x)-v_{\delta}(x)\right)^{+}}{|x|^{\beta}|x-y|^{\mu}|y|^{\alpha}}dxdy\\
			&\leq C_{\delta}\|v(y)^{\frac{n-(2\alpha+\mu)}{n-2}}\left(v(y)-v_{\delta}(y)\right)^{+}\|_{L^{\frac{2(n-1)}{2n-2-(2\alpha+\mu)}}(\partial\Sigma_{\delta}^{v})}\|v(x)^{\frac{n+2+2\lambda-(2\beta+\mu)}{n-2}}\left(v(x)-v_{\delta}(x)\right)^{+}\|_{L^{\frac{2n}{2n+2\lambda-(2\beta+\mu)}}(\Sigma_{\delta}^{v})}\\
			&\leq C_{\delta}\|v(x)\|_{L^{\frac{2n}{n-2}}(\Sigma_{\delta}^{v})}^{\frac{2\lambda+n+2-(2\beta+\mu)}{n-2}}\|\left(v(x)-v_{\delta}(x)\right)\|_{L^{\frac{2n}{n-2}}(\Sigma_{\delta}^{v})}\|v(y)\|_{L^{\frac{2(n-1)}{n-2}}(\partial\Sigma_{\delta}^{v})}^{\frac{n-(2\alpha+\mu)}{n-2}}\|\left(v(y)-v_{\delta}(y)\right)\|_{L^{\frac{2(n-1)}{n-2}}(\partial\Sigma_{\delta}^{v})}.
		\end{aligned}
	\end{equation}
	
	Hence, plugging \eqref{yi1}, \eqref{yi2}, \eqref{yi3} and \eqref{yi4} into \eqref{peng2} and taking $\varepsilon\rightarrow0$, then by Lebesgue’s dominated convergence theorem, Sobolev trace inequality and Sobolev inequality, it holds that
	\begin{equation}\nonumber
		\begin{aligned}
			&\int_{\Sigma_{\delta}}|\nabla\left(v-v_{\delta}\right)^{+}|^{2}dx\\
			&\leq C_{\delta}\left[\|v(x)\|_{L^{\frac{2n}{n-2}}(\Sigma_{\delta}^{v})}^{\frac{2\lambda+n+2-(2\beta+\mu)}{n-2}}\|v(y)\|_{L^{\frac{2(n-1)}{n-2}}(\partial\Sigma_{\delta}^{v})}^{\frac{n-(2\alpha+\mu)}{n-2}}+\left(\int_{\Sigma_{\delta}^{v}}\omega(x)^{\frac{2n}{(2\beta+\mu)-2\lambda}}dx\right)^{\frac{(2\beta+\mu)-2\lambda}{2n}}\left(\int_{\Sigma_{\delta}^{v}}v(x)^{\frac{2n}{n-2}}dx\right)^{\frac{2\lambda+4-(2\beta+\mu)}{2n}}\right.\\
			&~~~\left.+\left(\int_{\partial\Sigma_{\delta}^{v}}v(y)^{\frac{2(n-1)}{n-2}}dy\right)^{\frac{2-(2\alpha+\mu)}{2(n-1)}}\left(\int_{\partial\Sigma_{\delta}^{v}}z(y)^{\frac{2(n-1)}{2\alpha+\mu}}dy\right)^{\frac{2\alpha+\mu}{2(n-1)}}\right]\cdot\left(\int_{\Sigma_{\delta}^{v}}|\nabla\left(v-v_{\delta}\right)^{+}|^{2}dx\right).
		\end{aligned}
	\end{equation}
	which leads to \eqref{pe3}. This accomplishes the proof.
\end{proof}
Next, based on Lemma \ref{la1} and Lemma \ref{la2}, we give the following frequently useful lemma.	
\begin{lem}
	Under the hypothesis of Theorem \ref{app1}, there exists a non-negative constant $\delta_{0}$, such that for any $\delta\geq\delta_{0}$, $x\in\Sigma_{\delta}$ and $y\in\partial\Sigma_{\delta}$, we have
	\begin{equation}\label{dai1}
		\omega(x)\leq\omega(x^{\delta}),\ \ z(y)\leq z(y^{\delta}),\ \ v(x)\leq v_{\delta}(x).
	\end{equation} 
\end{lem}
\begin{proof}
	Actually, since $\omega(x)$, $z(y)$ and $v(x)$ possess better decay properties, then we can take $\delta_{0}$ sufficiently large, such that for $\delta>\delta_{0}$, it holds that
	\begin{equation}\nonumber
		\begin{aligned}
			C_{\delta}&\left(\|v(x)\|_{L^{\frac{2n}{n-2}}(\Sigma_{\delta}^{v})}^{\frac{2\lambda+n+2-(2\beta+\mu)}{n-2}}\|v(y)\|_{L^{\frac{2(n-1)}{n-2}}(\partial\Sigma_{\delta}^{v})}^{\frac{n-(2\alpha+\mu)}{n-2}}+\|\omega(x)\|_{L^{\frac{2n}{(2\beta+\mu)-2\lambda}}(\Sigma_{\delta}^{v})}\|v(x)\|^{\frac{2\lambda+4-(2\beta+\mu)}{2n}}_{L^{\frac{2n}{n-2}}(\Sigma_{\delta}^{v})}\right.\\
			&\left.+\|z(y)\|_{L^{\frac{2(n-1)}{n-2}}(\partial\Sigma_{\delta}^{v})}\|v(y)\|^{\frac{2-(2\alpha+\mu)}{2(n-1)}}_{L^{\frac{2(n-1)}{n-2}}(\partial\Sigma_{\delta}^{v})}\right)\leq \frac{1}{2}.
		\end{aligned}
	\end{equation}
	Further, according to Lemma \ref{la1} and Lemma \ref{la2}, for any $x\in\Sigma_{\delta}$ and $y\in\partial\Sigma_{\delta}$, we obtain $\omega(x)\leq\omega(x^{\delta})$, $z(y)\leq z(y^{\delta})$ and $v(x)\leq v_{\delta}(x)$.
\end{proof}

It should be noted that, the above analysis provides the starting point to move the plane $T_{\delta_{0}}$. Now we are ready to move the plane from the right to the left provided \eqref{dai1}. To state the process accurately, we write
\begin{equation}\label{pe4}
	\delta_{1}=\inf\left\{\delta|\omega(x)\leq\omega(x^{\delta}),~ z(y)\leq z(y^{\delta}),~ v(x)
	\leq v_{\delta}(x),~\forall x\in\Sigma_{\delta},~ \forall y\in\partial\Sigma_{\delta}\right\}.
\end{equation}

Then we deduce the following important result.
\begin{lem}\label{dai2}
	If $\delta_{1}>0$, then for any $x\in\Sigma_{\delta_{1}}$ and $y\in\partial\Sigma_{\delta_{1}}$, it holds that
	$\omega(x)\equiv\omega(x^{\delta_{1}})$, $z(y)\equiv z(y^{\delta_{1}})$ and $v(x)\equiv v_{\delta_{1}}(x)$.
\end{lem}
\begin{proof}
	Suppose that $\omega(x)\not\equiv\omega(x^{\delta_{1}})$, $z(y)\not\equiv z(y^{\delta_{1}})$ and $v(x)\not\equiv v_{\delta_{1}}(x)$. On one hand, by means of the continuity of $\omega(x)$, $v(x)$ and $z(y)$, for any $x\in\Sigma_{\delta_{1}}$ and $y\in\partial\Sigma_{\delta_{1}}$, we have 
	\begin{equation}
		\omega(x)\leq\omega(x^{\delta_{1}}),~ v(x)\leq v_{\delta_{1}}(x),~z(y)\leq z(y^{\delta_{1}}).
	\end{equation}
	
	Moreover, from the monotonicity of $g$ and $k$, we know
	\begin{equation}\nonumber
		\begin{aligned}
			x_{n}^{\lambda}\omega(x)k\left(|x|^{n-2}v(x)\right)v(x)^{\frac{2\lambda+n+2-(2\beta+\mu)}{n-2}}&=x_{n}^{\lambda}\omega(x)\frac{g\left(|x|^{n-2}v(x)\right)}{|x|^{2\lambda+n+2-(2\beta+\mu)}}\\
			&\leq x_{n}^{\lambda}\omega(x^{\delta_{1}})\frac{g\left(|x|^{n-2}v_{\delta_{1}}(x)\right)}{|x|^{2\lambda+n+2-(2\beta+\mu)}}\\
			&\leq x_{n}^{\lambda}\omega(x^{\delta_{1}})\frac{g\left(|x|^{n-2}v_{\delta_{1}}(x)\right)}{\left(|x|^{n-2}v_{\delta_{1}}(x)\right)^{\frac{2\lambda+n+2-(2\beta+\mu)}{n-2}}}v_{\delta}(x)^{\frac{2\lambda+n+2-(2\beta+\mu)}{n-2}}\\
			&\leq x_{n}^{\lambda}\omega(x^{\delta_{1}})\frac{g\left(|x|^{n-2}v_{\delta_{1}}(x)\right)}{\left(|x^{\delta_{1}}|^{n-2}v_{\delta_{1}}(x)\right)^{\frac{2\lambda+n+2-(2\beta+\mu)}{n-2}}}v_{\delta}(x)^{\frac{2\lambda+n+2-(2\beta+\mu)}{n-2}}\\
			&=x_{n}^{\lambda}\omega(x^{\delta_{1}})k\left(|x^{\delta_{1}}|^{n-2}v_{\delta_{1}}(x)\right)v_{\delta_{1}}(x)^{\frac{2\lambda+n+2-(2\beta+\mu)}{n-2}},
		\end{aligned}
	\end{equation}
	which immediately implies that
	\begin{equation}\nonumber
		-\Delta v(x)\leq -\Delta v(x^{\delta_{1}}),\ \ x\in\Sigma_{\delta_{1}}.
	\end{equation}
	Furthermore, according to the strong maximum principle, for any $x\in\Sigma_{\delta_{1}}$ and $y\in\partial\Sigma_{\delta_{1}}$, we obtain $\omega(x)<\omega(x^{\delta_{1}})$, $v(x)<v_{\delta_{1}}(x)$ and $z(y)<z(y^{\delta_{1}})$.
	
	On the other hand, when $\delta\rightarrow\delta_{1}$, we get $\frac{1}{|x|^{2n}}\chi_{\Sigma_{\delta}^{v}}\stackrel{a.e.}{\longrightarrow}0$ and $\frac{1}{|y|^{2(n-1)}}\chi_{\partial\Sigma_{\delta}^{v}}\stackrel{a.e.}{\longrightarrow}0$. Therefore, there exists $\tau>0$ such that for $\delta\in\left[\delta_{1}-\tau,\delta_{1}\right]$, $\frac{1}{|x|^{2n}}\chi_{\Sigma_{\delta}^{v}}\leq \frac{1}{|x|^{2n}}\chi_{\Sigma_{\delta_{1}-\tau}^{v}}$ and $\frac{1}{|y|^{2(n-1)}}\chi_{\partial\Sigma_{\delta}^{v}}\leq \frac{1}{|y|^{2(n-1)}}\chi_{\partial\Sigma_{\delta_{1}-\tau}^{v}}$. Here, applying Lebesgue’s dominated convergence theorem, when $\delta\rightarrow\delta_{1}$, we know
	\begin{equation}\nonumber
		\int_{\partial\Sigma_{\delta}^{v}}\frac{1}{|x|^{2n}}dx\rightarrow0,\ \ ~~~~ \int_{\partial\Sigma_{\delta}^{v}}\frac{1}{|y|^{2(n-1)}}dy\rightarrow0.
	\end{equation}
	
	Finally, using the decay properties of $\omega(x)$, $z(y)$ and $v(x)$, we point out that there exists $\tau_{0}$ such that for any $\delta\in\left[\delta_{1}-\tau_{0},\delta_{1}\right]$, it holds 
	\begin{equation}\nonumber
		\begin{aligned}
			C_{\delta}&\left(\|v(x)\|_{L^{\frac{2n}{n-2}}(\Sigma_{\delta}^{v})}^{\frac{2\lambda+n+2-(2\beta+\mu)}{n-2}}\|v(y)\|_{L^{\frac{2(n-1)}{n-2}}(\partial\Sigma_{\delta}^{v})}^{\frac{n-(2\alpha+\mu)}{n-2}}+\|\omega(x)\|_{L^{\frac{2n}{(2\beta+\mu)-2\lambda}}(\Sigma_{\delta}^{v})}\|v(x)\|^{\frac{2\lambda+4-(2\beta+\mu)}{2n}}_{L^{\frac{2n}{n-2}}(\Sigma_{\delta}^{v})}\right.\\
			&\left.+\|z(y)\|_{L^{\frac{2(n-1)}{n-2}}(\partial\Sigma_{\delta}^{v})}\|v(y)\|^{\frac{2-(2\alpha+\mu)}{2(n-1)}}_{L^{\frac{2(n-1)}{n-2}}(\partial\Sigma_{\delta}^{v})}\right)< \frac{1}{2}.
		\end{aligned}
	\end{equation}
	Therefore, by using Lemma \ref{la1} and Lemma \ref{la2}, we have for any $\delta\in\left[\delta_{1}-\tau_{0},\delta_{1}\right]$,
	\begin{equation}\nonumber
		\omega(x)\leq\omega(x^{\delta}),~ v(x)\leq v_{\delta}(x),~z(y)\leq z(y^{\delta}),\ \ x\in\Sigma_{\delta},~~ y\in\partial\Sigma_{\delta},
	\end{equation}
	which contradicts the definition of $\delta_{1}$. The proof is finished.
\end{proof}

It's clear that, the proof of Theorem \ref{app1} relies on the above results.

\textbf{Proof of Theorem \ref{app1}.} In fact, we can move the plane $T_{\delta}$ from $\infty$ to the left, and continue this proof until $\delta=\delta_{1}$. If $\delta_{1}>p_{1}$, then we obtain
\begin{equation}\nonumber
	\omega(x)\equiv\omega(x^{\delta_{1}}),~z(y)\equiv z(y^{\delta_{1}}),~v(x)\equiv v_{\delta_{1}}(x),\ \	x\in\Sigma_{\delta_{1}},~y\in\partial\Sigma_{\delta_{1}}
\end{equation}
However, this is impossible. Therefore, we get $\delta_{1}\leq p_{1}$. Similarly, we can move the plane from $-\infty$ to the right as we did in the previous discussion, then we derive $\delta_{1}^{\prime}$ and $\delta_{1}^{\prime}\geq p_{1}$. Furthermore, we conclude $\delta_{1}=\delta_{1}^{\prime}=p_{1}$. Since $x_{1}$ direction can be taken arbitrarily, then the fact implies that, $\omega(x)$, $v(x)$ and $z(y)$ are symmetric with respect to any plane, which is passing through $p$ and perpendicular to $x_{i}~\left(i=1,2...n-1\right)$ axis. 

Finally, for any $p\in\partial\mathbb{R}^{n}_{+}$, $v(x)$, $\omega(x)$ and $z(y)$ are symmetric about the plane that passing through $p$ and is parallel to $x_{n}$ axis. Hence we derive that $u(x)$ and $m(x)$ depend only on $x_{n}$ and $\eta$ is constant. This proof is accomplished.
$\hfill{} \Box$

Next, we will pay attention to study the existence of positive solution for Hartree type equations \eqref{main1}. By virtue of Theorem \ref{app1}, we give the proof of Corollary \ref{rem}.

\textbf{Proof of Corollary \ref{rem}.} Naturally, according to Theorem \ref{app1}, we conclude that $m(x)$, $u(x)$ depend only $x_{n}$ and $\eta(y)$ is constant function. As a consequence, we can rewrite the equations \eqref{main1} in the following form
\begin{equation}\nonumber
\left\lbrace 
\begin{aligned}
&-\frac{d^{2}u(x_{n})}{dx_{n}^{2}}=\left(\int_{\partial\mathbb{R}^{n}_{+}}\frac{F\left(u(0)\right)}{|x-y|^{\mu}}dy\right)x_{n}^{\lambda}g\left(u(x_{n})\right),\ \ x_{n}>0,\\
&-\frac{\partial u}{\partial x_{n}}(0)=\left(\int_{\mathbb{R}^{n}_{+}}\frac{G\left(u(x_{n})\right)x_{n}^{\lambda}}{|x-y|^{\mu}}dx\right)f\left(u(0)\right),
\end{aligned}
\right.
\end{equation}

It is worth noting that the first equation yield $u(x)$ is concave function. Moreover, from the later equation, we have
\begin{equation}\nonumber
\frac{d u}{d x_{n}}(0)=\left(\int_{\mathbb{R}^{n}_{+}}\frac{G\left(u(x_{n})\right)x_{n}^{\lambda}}{|x-y|^{\mu}}dx\right)f\left(u(0)\right)\geq0.
\end{equation}
Thus, we conclude that $u(x)$ is concave and decreasing unless $u\equiv \tilde{c}$ with $F(\tilde{c})=G(\tilde{c})=0$. On the one hand, if $u$ is strictly decreasing with respect to $x_{n}$, then we have $\frac{d u}{d x_{n}}(0)\leq0$, which is impossible. The fact implies that this case will not happen. Hence we immediately deduce the desired result. The proof is accomplished.
$\hfill{} \Box$

\end{document}